\documentclass[a4paper,12pt,final]{article}%
\usepackage{graphicx}
\usepackage{amsmath}
\usepackage{amsfonts}
\usepackage{amssymb}
\usepackage{version}%
\providecommand{\U}[1]{\protect\rule{.1in}{.1in}}
\newtheorem{theorem}{Theorem}

\newtheorem{corollary}[theorem]{Corollary}

\newtheorem{definition}[theorem]{Definition}
\newtheorem{example}[theorem]{Example}

\newtheorem{lemma}[theorem]{Lemma}

\newtheorem{proposition}[theorem]{Proposition}
\newtheorem{remark}[theorem]{Remark}

\newenvironment{proof}[1][Proof]{\textbf{#1.} }{\  \rule{0.5em}{0.5em}}
\newcommand{\dcoset}{/\! \!/}

\usepackage{hyperref}

\begin{document}

\title{Wiener's theorem for positive definite functions on hypergroups}
\author{Walter R Bloom\\{\small Murdoch University}\\{\small Perth, Western Australia}\\{\small Email: w.bloom@murdoch.edu.au}
\and John J.F. Fournier\\{\small University of British Columbia} \\{\small Vancouver, Canada}\\{\small Email: fournier@math.ubc.ca}
\and Michael Leinert\\{\small {University of Heidelberg}}\\{\small Heidelberg, Germany}\\{\small Email: leinert@math.uni-heidelberg.de}}
\date{}
\maketitle

\begin{abstract}
The following theorem\ on the circle group $\mathbb{T}$ is due to Norbert
Wiener: If $f\in L^{1}\left(  \mathbb{T}\right)  $ has non-negative Fourier
coefficients and is square integrable on a neighbourhood of the identity, then
$f\in L^{2}\left(  \mathbb{T}\right)  $. This result has been extended to even
exponents including $p=\infty$, but shown to fail for all other $p\in\left(
1,\infty\right]  .$ All of this was extended further (appropriately
formulated) well beyond locally compact abelian groups. In this paper we prove
Wiener's theorem for even exponents for a large class of commutative
hypergroups. In addition, we present examples of commutative hypergroups for
which, in sharp contrast to the group case, Wiener's theorem holds for
\textit{all} exponents $p\in\left[  1,\infty\right]  $. For these hypergroups
and the Bessel-Kingman hypergroup with parameter $\frac{1}{2}$ we characterise
those locally integrable functions that are of positive type and
square-integrable near the identity in terms of amalgam spaces.

\end{abstract}

%The next 3 lines put a header with date stamp
%on each page after the first one.
%The date used records that last time
%that the file was edited.
%That requires changing the next line appropriately.

%\markboth{\now}{\now} \pagestyle{myheadings}

\section{Introduction}

On the unit circle $\mathbb{T}$\ consider the following statement: If an
integrable function on $\mathbb{T}$ has non-negative Fourier coefficients and
is $p-$integrable on some neighbourhood of the identity, then $f$ is
$p-$integrable on all of $\mathbb{T}$. For $p=2$ this is a theorem of Norbert
Wiener. It was then shown to hold for all even $p\in\mathbb{N}$ and $p=\infty
$, but to fail for all other $p\in\left(  1,\infty\right]  $ \cite{Wainger},
\cite{Shapiro}. All of this was extended (appropriately formulated)
successively to compact abelian \cite{Rains}, locally compact abelian
\cite{Fournier1} and finally $IN$-groups \cite{Leinert} (groups having at
least one relatively compact neighbourhood of the identity invariant under
inner automorphisms).
%For the reasons given in \cite{Leinert}, p.1,
Since, in the original formulation, Wiener's theorem does not extend to
non-compact groups (it fails even for the real line), the results on
non-compact groups $G$ are formulated with $L^{p}\left(  G\right)  $ replaced
by the amalgam space $\left(  L^{p},\ell^{\infty}\right)  \left(  G\right)  $.
(for compact groups this is no change, as $(L^{p},\ell^{\infty})\left(
G\right)  =L^{p}\left(  G\right)  $ in this case). Related information can be
found in [13], p. 1.

In Section \ref{Wiener} of this paper
%In Corollary \ref{Amalgam equivalence}\ below
we extend the positive result to a large class of commutative hypergroups,
namely those where the product of bounded continuous positive definite
functions is itself positive definite (see Corollary \ref{Amalgam equivalence}%
\ below). In particular this applies to
%holds for
strong hypergroups.

In Section \ref{hypergroups}\
%we show (Theorem \ref{transforms})
we consider Bessel-Kingman hypergroups. These are strong hypergroups, so the
results of Section 2 apply to them. For the motion hypergroup,
%motion hypergroup
\textit{i.e.} the Bessel-Kingman hypergroup with $\alpha=\frac{1}{2}$, we show
(Theorem \ref{transforms}) that for $p=2$ there is a characterization like the
one in \cite{Fournier1}
%like that of Theorem 3.1
of positive definite functions that are square integrable near the identity.
%characterization like that of Theorem 3.1
%in \cite{Fournier1}
Since the proof (following \cite{Fournier1}) makes use of results about
Fourier transforms, duality and interpolation for amalgam spaces defined via
certain tilings, we need to show that on this hypergroup the norms for these
spaces are equivalent to amalgam norms defined using translations. For groups
this equivalence is well known (see \cite{Feichtinger}), but for hypergroups
this is not clear.
%we are unaware of any previous work on this question for hypergroups that are not groups.
We obtain some results on translation, convolution and the Fourier transform
for amalgam spaces on the motion hypergroup; these are needed for the proof of
Theorem \ref{transforms}. We also compare our amalgam norms with some other
ones, including those in \cite{CowlingMedaPasquale}.

Finally in Section \ref{Countable} we look at the countable non-discrete
hypergroups considered in \cite{DunklRamirez} and \cite{Spector}. We prove the
analogue of Theorem \ref{transforms} and show that for these hypergroups,\ in
sharp contrast to the group case, Wiener's theorem holds for \textit{all}
exponents $p\in\left[  1,\infty\right]  $; see Theorem \ref{LocalGlobal} and
Corollary \ref{CompactCase} below.

\section{Wiener's theorem for even $p\in\mathbb{N}$ or $p=\infty$}

\label{Wiener}

Let $K$ be a
%strong
hypergroup with Haar measure $\omega_{K}$.
%and Plancherel measure $\pi_{K}$.
%We call
%$f\in L^{1}\left(  K\right)  $ (pointwise)
In the following any unexplained notation will be taken from
\cite{Bloom/Heyer}. Recall that, although the product of two elements, say
$x,y$ of $K$, might not be defined, the convolution of the unit point masses
$\varepsilon_{x}$ and $\varepsilon_{y}$ is defined. When the integral of a
function $f$ on $K$ against the measure $\varepsilon_{x}\ast\varepsilon_{y}$
is defined, that integral is denoted by $f(x\ast y)$. We recall the definition
of positive definiteness on hypergroups (\cite{Bloom/Heyer}, Definition 4.1.1).

\begin{definition}
\label{definePD} A function $f$ on $K$ is called \emph{positive definite} if
it is measurable and locally bounded, and%
\[
\sum_{i=1}^{n}\sum_{j=1}^{n}c_{i}\overline{c_{j}}f\left(  x_{i}\ast x_{j}%
^{-}\right)  \geq0
\]
for all choices of $c_{i}\in\mathbb{C},\,x_{i}\in K$ and $n\in\mathbb{N}.$
\end{definition}

The set of continuous positive definite functions will be denoted by $P\left(
K\right)  .$\ Note that, unlike for groups, there are hypergroups where such
functions are not necessarily bounded (see \cite{Bloom/Heyer}, p. 268 or
Remark \ref{AmalgamExtension}(a) below). The subset of bounded functions in
$P(K)$ is denoted by $P_{b}(K)$.
%We observe (see \cite{Bloom/Heyer}, Proposition 4.1.18) that for
%functions $f\in L^{1}\left(  K\right)  \cap C\left(  K\right)  $ the condition
%$f\in P_{b}\left(  K\right)  $\ is equivalent to $\overset{\wedge}{f}\geq0$
%as, by the assumption on $K$, the Plancherel measure $\pi$\ has support the
%whole dual.

When $f,g$ and $h$ are functions on $K$, the notation $f(g\ast h)$ will mean
the pointwise product of the function $f$ with the convolution $g\ast h$,
rather than meaning the integral of $f$ against a measure $g\ast h$ as in the
notation $f(x\ast y)$ above. We sometimes also write $(g\ast h)f$ or
$f\cdot(g\ast h)$ (and this extends to cases where $g$ is a measure).

\begin{definition}
\label{positivetype}A locally integrable function $f$ is said to be of
\emph{positive type} if
\[
\int f\cdot(g\ast g^{\ast})\,d\omega_{K}\geq0
\]
for every $g\in C_{c}(K),$ where $g^{\ast}\left(  x\right)  :=\Delta\left(
x^{-}\right)  g^{\sim}\left(  x\right)  ,g^{\sim}:=\overline{g^{-}}$ and
$g^{-}\left(  x\right)  :=g\left(  x^{-}\right)  $ for all $x\in K$.
\end{definition}

For continuous $f$ this amounts to saying that $f$ is positive definite (see
\cite{Bloom/Heyer}, Lemma 4.1.4; when $K$ is not unimodular, the function
$g^{\sim}\ $in part (iii) of that lemma should be replaced by the function
$g^{\ast}$).
%Recall that for commutative hypergroups, requiring that a function~$f$
%in $L^1(K)$ be of positive type
%is equivalent to requiring that~$\hat f \ge 0$
%on the support of~$\pi_K$.
In particular, if $K$ is discrete the notions \textquotedblleft of positive
type\textquotedblright\ and \textquotedblleft positive definite" coincide.

\begin{remark}
If $K$ is any non-discrete hypergroup, there exist lower semicontinuous
functions of positive type in $L^{1}(K)$ that are unbounded near the identity
and hence don't belong to $P(K)$. To see this, note that using the outer
regularity of $\omega_{K}$ for the null set $\{e\}$ there is a decreasing
sequence of symmetric neighbourhoods $U_{n}$ with $\omega_{K}\left(
U_{n}\right)  \rightarrow0$, and we may assume $\omega_{K}\left(
U_{n}\right)  <1/n$. Let $f=\sum\lambda_{n}\mathbf{1}_{U_{n}}\ast
\mathbf{1}_{U_{n}}$ where $\lambda_{n}=1/\left(  n\,\omega_{K}\left(
U_{n}\right)  \right)  $ and $\mathbf{1}_{U_{n}}$ is the indicator function of
$U_{n}$. Now%
\[
\left\Vert f\right\Vert _{1}=\sum\lambda_{n}\omega_{K}\left(  U_{n}\right)
^{2}\leq\sum\frac{1}{n^{2}}<\infty.
\]
Being the supremum of continuous functions, $f$ is lower semicontinuous, and
we have%
\[
f(e)=\sum\lambda_{n}\omega_{K}\left(  U_{n}\right)  =\sum\frac{1}{n}=\infty
\]
so $f$ is unbounded near $e$. Since $\mathbf{1}_{U_{n}}=\left.  \mathbf{1}%
_{U_{n}}\right.  ^{\sim}$, $f$ is of positive type.
%If $K$ is discrete, however, the notions "of positive type" and "positive definite" coincide.

\end{remark}

On several occasions in this paper we use that if $f$ is a function of
positive type and $h$ is is a real-valued continuous function with compact
support, then $h\ast f\ast h^{-}$ is of positive type. This can be seen from
the definition using \cite{Bloom/Heyer}, (1.4.23), that is%
\begin{equation}
\int\left(  f\ast h\right)  g\text{ }d\omega_{K}=\int f\cdot\left(  g\ast
h^{-}\right)  \text{ }d\omega_{K} \label{BH(1.4.23)}%
\end{equation}
and its left-hand version%
\begin{equation}
\int\left(  h\ast f\right)  g\text{ }d\omega_{K}=\int f\cdot\left(  \left(
\Delta^{-}h^{-}\right)  \ast g\right)  \text{ }d\omega_{K}
\label{Left-hand version}%
\end{equation}
which has a similar proof. (Note that $\Delta^{-}h^{-}=h^{\ast}$ since $h$ is
real-valued.) In the special case when $f\in L^{1}(K)$ and the hypergroup $K$
is commutative, we can also see this using the Fourier transform.
%use the Fourier transform instead.

\begin{remark}
\label{Positive type}Let $K$ be a commutative hypergroup. A function $f\in
L^{1}(K)$ is of positive type if and only if $\overset{\wedge}{f}\geq0$ on the
support of the Plancherel measure $\pi_{K}$.
\end{remark}

\begin{proof}
(a) Let $f\in L^{1}(K)$ be of positive type and let $\chi\in\mathrm{supp\,}%
\pi_{K}$. By \cite{Bloom/Heyer}, 4.1.22 there is net $(f_{\iota})$ in
$C_{c}(K)$ such that $f_{\iota}\ast f_{\iota}^{\sim}\rightarrow\chi$ uniformly
on compact sets. We may assume that $f_{\iota}\ast f_{\iota}^{\sim}(e)=1$ for
all $\iota$. For $\varepsilon>0$ choose a compact $C\subset K$ such that
$\int_{K\backslash C}|f|\,d\omega_{K}<\varepsilon$. Since $\left\vert
\chi\right\vert \leq1$ and $\left\vert f_{\iota}\ast f_{\iota}^{\sim
}\right\vert \leq1$ (for the second inequality, note that $f_{\iota}\ast
f_{\iota}^{\sim}\in P_{b}\left(  K\right)  $ by \cite{Bloom/Heyer}, Lemma
4.1.5(b) and the bound follows from \cite{Bloom/Heyer}, Lemma 4.1.3(g)), we
have that
\[
\left\vert \overset{\wedge}{f}(\chi)-\int f\cdot(f_{\iota}\ast f_{\iota}%
^{\sim})\,d\omega_{K}\right\vert \leq2\varepsilon+\int_{C}|f|\left\vert
\chi-(f_{\iota}\ast f_{\iota}^{\sim})\right\vert \,d\omega_{K}<3\varepsilon
\]
for suitable $\iota$. By the assumption on $f$ we have $\int f\cdot(f_{\iota
}\ast f_{\iota}^{\sim})\,d\omega_{K}\geq0$ (note that $f_{\iota}^{\sim
}=f_{\iota}^{\ast}$ since $K$ is unimodular), and hence $\overset{\wedge
}{f}(\chi)\geq0$.

(b) Suppose $\overset{\wedge}{f}\geq0$ on $\mathrm{supp\,}\pi_{K}$ and let
$g\in C_{c}(K)$. We have using (\ref{BH(1.4.23)}) and Plancherel's theorem%
\[
\int f\cdot\left(  g\ast g^{\sim}\right)  \,d\omega_{K}=\int\left(
f\ast\overline{g}\right)  \overline{\overline{g}}\,d\omega_{K}=\int\hat{f}%
\hat{\overline{g}}\overline{\hat{\overline{g}}}\,d\pi_{K}=\int\hat
{f}\left\vert \hat{\overline{g}}\right\vert ^{2}\,d\pi_{K}\geq0
\]

\end{proof}

As in \cite{Bloom/Heyer}, p.\thinspace8 the set of all probability measures on
$K$ will be denoted by $M^{1}\left(  K\right)  $.

\begin{lemma}
\label{L^1case} \label{Property*}
%Let $p=2k$ where $k\in\mathbb{N}.$
Let $K$ be a commutative hypergroup. For every relatively compact
neighbourhood $U$ of the identity there is a constant $C_{U}>0$ such that
%and $f\in L^{1}\left(  K\right)  $ with $\overset{\wedge}{f}\geq0$
%on $K^{\wedge}$\ we have%
\begin{equation}
\int g\cdot\left(  \mu\ast\mathbf{1}_{U}\right)  \,d\omega_{K}\leq C_{U}\int
g\mathbf{1}_{U}\,d\omega_{K} \label{*}%
\end{equation}
for all choices of $\mu\in M^{1}\left(  K\right)  $ and all non-negative $g\in
P_{b}\left(  K\right)  $.
\end{lemma}

\begin{proof}
By Theorem 4.1.13 of \cite{Bloom/Heyer} we may write $g(x)$ as a coefficient
of a cyclic representation $D$ of the hypergroup $K$ on a Hilbert space
$\mathcal{H}$, that is there is a cyclic vector $\mathbf{u}\in\mathcal{H}$
such that%
\[
g(x)=\left\langle D(x)\mathbf{u},\mathbf{u}\right\rangle _{\mathcal{H}}%
\]
for all $x\in K$.

Choose a
%symmetric
relatively compact neighbourhood $V$ of $e$ such that
\[
\bigcup\left\{  \mathrm{supp\,}\left(  \varepsilon_{x^{-}}\ast\varepsilon
_{y}\right)  :x,y\in V\right\}  \subset U
\]
and $\omega_{K}\left(  V\right)  \leq1;$ these conditions guarantee that%
\begin{equation}
h:=\mathbf{1}_{V}^{\sim}\ast\mathbf{1}_{V}\leq\mathbf{1}_{U}. \label{nesting}%
\end{equation}
Since $h\in C_{c}^{+}\left(  K\right)  $ with $h\left(  e\right)  >0$ and $U$
is relatively compact, there exist $x_{1},x_{2},...,x_{n}\in K$ and
$\lambda_{1},\lambda_{2},...,\lambda_{n}>0$ such that $\mathbf{1}_{U}\leq
\sum_{i=1}^{n}\lambda_{i}\, \tau_{x_{i}}h,$ where%
\[
\tau_{x_{i}}h\left(  y\right)  =h\left(  x_{i}\ast y\right)
\]
is the $x_{i}-$translate of $h.$

%\begin{align*}
Let $\nu=\sum_{i=1}^{n}\lambda_{i}\varepsilon_{x_{i}}$. Then
\begin{align*}
\int g\cdot(\mu\ast\mathbf{1}_{U})\,d\omega_{K}  &  \leq\int g\left(  \mu
\ast\left(  \sum_{i=1}^{n}\lambda_{i}\, \tau_{x_{i}}h\right)  \right)
\,d\omega_{K}\\
&  =\left\langle D(\mu\ast\nu^{-}\ast h)\mathbf{u},\mathbf{u}\right\rangle
_{\mathcal{H}}\\
&  =\left\langle D(\mu\ast\nu^{-}\ast\mathbf{1}_{V})\mathbf{u},D(\mathbf{1}%
_{V})\mathbf{u}\right\rangle _{\mathcal{H}}\\
&  =\left\langle D(\mu\ast\nu^{-})D(\mathbf{1}_{V})\mathbf{u},D(\mathbf{1}%
_{V})\mathbf{u}\right\rangle _{\mathcal{H}}\\
&  \leq\Vert D(\mu\ast\nu^{-})\Vert_{B\left(  \mathcal{H}\right)  }\, \Vert
D(\mathbf{1}_{V})\mathbf{u}\Vert_{\mathcal{H}}^{2}\\
&  \leq\Vert\nu\Vert\int hg\, \,d\omega_{K}%
\end{align*}
since $\Vert\mu\Vert=1$, and since
\[
\Vert D(\mathbf{1}_{V})\mathbf{u}\Vert_{\mathcal{H}}^{2}=\left\langle
D(\mathbf{1}_{V})\mathbf{u},D(\mathbf{1}_{V})\mathbf{u}\right\rangle
_{\mathcal{H}}=\left\langle D(\mathbf{1}_{V})^{\ast}D(\mathbf{1}%
_{V})\mathbf{u},\mathbf{u}\right\rangle _{\mathcal{H}}%
\]%
\[
=\left\langle D(\mathbf{1}_{V}^{\sim}\ast\mathbf{1}_{V})\mathbf{u}%
,\mathbf{u}\right\rangle _{\mathcal{H}}=\left\langle D(h)\mathbf{u}%
,\mathbf{u}\right\rangle _{\mathcal{H}}=\int hg\, \,d\omega_{K}.
\]
So, letting $C_{U}=\Vert\nu\Vert$, we have that
\[
\int g\cdot\left(  \mu\ast\mathbf{1}_{U}\right)  \,d\omega_{K}\leq C_{U}\int
g\mathbf{1}_{U}\,d\omega_{K}.
\]

\end{proof}

%where for the third inequality we have appealed to (\ref{(c1)})\ noting that,
%by assumption on the hypergroup, $P_{b}\left(  K\right)  $ is closed under
%pointwise multiplication from which it follows that $\left|  f\right|  ^{p}\in
%P_{b}\left(  K\right)  \cap L^{1}\left(  K\right)  $ (recall that $p$ is an
%even integer) and that $\left\|  \varepsilon_{x}\ast\varepsilon_{x_{i}%
%}\right\|  =1.$ Writing $C_{U}=\left(  \sum_{i=1}^{n}\lambda_{i}\right)
%^{1/p}$ we have, for $\mu\in M^{\left(  1\right)  }\left(  K\right)  $,%
%\begin{align*}
%\int\left|  f\right|  ^{p}\,\mu\ast\mathbf{1}_{U}\,d\omega &  =\int\left(
%\left|  f\right|  ^{p}\int\tau_{x}\mathbf{1}_{U}\,d\mu\left(  x\right)
%\right)  d\omega_{K}\\
%&  =\int\left(  \int\left|  f\right|  ^{p}\,\tau_{x}\mathbf{1}_{U}%
%\,d\omega_{K}\right)  d\mu\left(  x\right) \\
%&  \leq C_{U}^{p}\int\left(  \int\left|  f\right|  ^{p}\mathbf{1}_{U}%
%\,d\omega_{K}\right)  d\mu\left(  x\right) \\
%&  \leq C_{U}^{p}\int\left|  f\right|  ^{p}\mathbf{1}_{U}\,d\omega_{K}%
%\end{align*}

%We apply this to powers~$|f|^p$ when~$p$ is even.
%now show that the following condition is satisfied:

\begin{corollary}
\label{special case}
%Suppose that~$f$ is a continuous bounded function
%on a commutative strong hypergroup~$K$.
Let $K$ be a commutative hypergroup such that $P_{b}(K)\cdot P_{b}(K)\subset
P_{b}(K)$ and let $p\in\mathbb{N}$ be even. For every relatively compact
neighbourhood $U$ of the identity there is a constant $C_{U}>0$ such that for
all choices of $\mu\in M^{1}\left(  K\right)  $ and $f\in P_{b}\left(
K\right)  $
%with $\overset{\wedge}{f}\geq0$ on $K^{\wedge}$\%
\begin{equation}
\int\left\vert f\right\vert ^{p}\cdot\left(  \mu\ast\mathbf{1}_{U}\right)
\,d\omega_{K}\leq C_{U}\int\left\vert f\right\vert ^{p}\mathbf{1}_{U}%
\,d\omega_{K}. \label{(c2)}%
\end{equation}
$
%\cap L^{1}\left(  K\right)
$
\end{corollary}

\begin{proof}
Let $p\in\mathbb{N}$ be even. Since $f\in P_{b}(K)$, the same is true for
$\overline{f}$. It follows that
\[
|f|^{p}=(\overline{f}f)^{p/2}\in P_{b}(K)
\]
and it is also positive. Inserting $g=|f|^{p}$ in inequality (\ref{*}) yields
the inequality (\ref{(c2)}).
\end{proof}

\begin{remark}
\label{PositiveDefiniteproduct}We remind the reader that for strong
hypergroups,%
\[
P_{b}\left(  K\right)  \cdot P_{b}\left(  K\right)  \subset P_{b}\left(
K\right)  .
\]
(Use Bochner's theorem to write two functions $f$ and $g$ in $P_{b}\left(
K\right)  $ as inverse transforms of two nonnegative measures $\mu,\nu$
respectively on $K^{\wedge}$. Then $fg$ is the inverse transform of $\mu
\ast\nu$ and hence belongs to $P_{b}\left(  K\right)  $ as well.) In
particular, Corollary \ref{special case} and much of what follows holds for
all strong hypergroups.
\end{remark}

We now extend inequality (\ref{(c2)}) to integrable functions $f$ of positive type.

\begin{corollary}
\label{bad} Let $K$ be a commutative hypergroup such that $P_{b}(K)\cdot
P_{b}(K)\subset P_{b}(K)$ and take $p\in\mathbb{N}$ to be even. For every
relatively compact neighbourhood $U$ of the identity there is a constant
$C_{U}>0$ such that for all choices of $\mu\in M^{1}\left(  K\right)  $ and
$f\in L^{1}\left(  K\right)  $
%with $\overset{\wedge}{f}\geq0$ on $K^{\wedge}$\
of positive type (equivalently: $f\in L^{1}(K)$ with $\overset{\wedge}{f}%
\geq0$ on $\mathrm{supp\,}\pi_{K}$) we have%
\begin{equation}
\int\left\vert f\right\vert ^{p}\cdot\left(  \mu\ast\mathbf{1}_{U}\right)
\,d\omega_{K}\leq C_{U}\int\left\vert f\right\vert ^{p}\mathbf{1}_{U}%
\,d\omega_{K}. \label{pNorm}%
\end{equation}

\end{corollary}

\begin{proof}
Let $f$ be such a function with $\int\left\vert f\right\vert ^{p}\,
\mathbf{1}_{U}\,d\omega_{K}<\infty$ and write $f_{\iota}=k_{\iota}\ast f\ast
k_{\iota}^{-}$ where $k_{\iota}\in C_{c}^{+}\left(  K\right)  ,\int k_{\iota
}\,d\omega_{K}=1$ and $\mathrm{supp\,}k_{\iota}\downarrow\left\{  e\right\}
.$ (If $K$ is first countable, then this approximate identity can in fact be
chosen to be a sequence.) Clearly $f_{\iota}$ is bounded, continuous and
integrable. Since $f_{\iota}$ is of positive type (see the paragraph
immediately preceding Remark \ref{Positive type}), it is also in $P_{b}\left(
K\right)  $. Now the values of $f_{\iota}$ on $U$ depend on the values of $f$
on a slightly larger neighbourhood $U^{\prime},$ and we cannot rule out
\textit{a priori} the possibility that $\int\left\vert f\right\vert ^{p}\,
\mathbf{1}_{U^{\prime}}\,d\omega_{K}=\infty.$ For this technical reason we
first use a compact neighbourhood $W$ of $e$ contained in the interior of $U.$

For sufficiently large $\iota$ the values of $f_{\iota}$ on $W$ only depend on
the values of $f$ on $U,$ and we have%
\begin{equation}
\left\Vert \left(  f-f_{\iota}\right)  \mathbf{1}_{W}\right\Vert _{p}%
\leq\left\Vert f\mathbf{1}_{U}-k_{\iota}\ast\left(  f\mathbf{1}_{U}\right)
\ast k_{\iota}^{-}\right\Vert _{p}\rightarrow0 \label{NormConvergence}%
\end{equation}
since $f\mathbf{1}_{W}=f\mathbf{1}_{U}\mathbf{1}_{W}$ and $f_{\iota}%
\mathbf{1}_{W}=\left[  k_{\iota}\ast\left(  f\mathbf{1}_{U}\right)  \ast
k_{\iota}^{-}\right]  \mathbf{1}_{W}$ for sufficiently large $\iota.$ We also
have%
\begin{equation}
\left\Vert f_{\iota}-f\right\Vert _{1}\rightarrow0 \label{Sequence}%
\end{equation}
and we can extract a sequence $\left(  f_{n}\right)  \ $from $\left(
f_{\iota}\right)  $ satisfying both (\ref{NormConvergence}) and
(\ref{Sequence}), and (if necessary, passing to a subsequence thereof)
converging pointwise $a.e.$ to $f.$ Using Fatou's lemma we obtain%
\begin{align*}
\int\left\vert f\right\vert ^{p}\cdot\left(  \mu\ast\mathbf{1}_{W}\right)
\,d\omega_{K}  &  \leq\liminf_{n} \int\left\vert f_{n}\right\vert ^{p}\,
\mu\ast\mathbf{1}_{W}\,d\omega_{K}\\
&  \leq C_{W}\liminf_{n} \int\left\vert f_{n}\right\vert ^{p}\, \mathbf{1}%
_{W}\,d\omega_{K}\\
&  \leq C_{W}\int\left\vert f\right\vert ^{p}\, \mathbf{1}_{W}\,d\omega_{K}%
\end{align*}
where, for the middle inequality, we have appealed to (\ref{(c2)})$,$ and the
last inequality follows from (\ref{NormConvergence})$.$ Choose $x_{1}%
,x_{2},...,x_{n}\in K$ and $\lambda_{1},\lambda_{2},...,\lambda_{n}>0$ such
that $\mathbf{1}_{U}\leq\sum_{i=1}^{n}\lambda_{i}\, \tau_{x_{i}}\mathbf{1}%
_{W}.$ We then have%
\begin{align}
\int\left\vert f\right\vert ^{p}\cdot\left(  \mu\ast\mathbf{1}_{U}\right)
\,d\omega_{K}  &  \leq\sum_{i=1}^{n}\lambda_{i}\int\left\vert f\right\vert
^{p}\cdot\left(  \mu\ast\tau_{x_{i}}\mathbf{1}_{W}\right)  \,d\omega
_{K}\nonumber\\
&  =\sum_{i=1}^{n}\lambda_{i}\int\left\vert f\right\vert ^{p}\cdot\left(
\mu\ast\varepsilon_{x_{i}^{-}}\ast\mathbf{1}_{W}\right)  \,d\omega
_{K}\nonumber\\
&  \leq C_{W}\left(  \sum_{i=1}^{n}\lambda_{i}\right)  \int\left\vert
f\right\vert ^{p}\, \mathbf{1}_{W}\,d\omega_{K}\nonumber\\
&  \leq\left(  \sum_{i=1}^{n}\lambda_{i}\right)  C_{W}\int\left\vert
f\right\vert ^{p}\, \mathbf{1}_{U}\,d\omega_{K} \label{(c3)}%
\end{align}
and this ends the proof of the corollary.
\end{proof}

%\begin{corollary}
%\label{GeneralCase}Let $p=2k$ where $k\in\mathbb{N}.$ For every relatively
%compact neighbourhood $U$ of the identity there is a constant $C_{U}$ such
%that for all $\mu\in M^{\left(  1\right)  }\left(  K\right)  $ and $f\in
%L^{1}\left(  K\right)  $ with $\overset{\wedge}{f}\geq0$ on $K^{\wedge}$\ we
%have%
%\[
%\int\left|  f\right|  ^{p}\,\mu\ast\mathbf{1}_{U}\,d\omega_{K}\leq C_{U}%
%\int\left|  f\right|  ^{p}\mathbf{1}_{U}\,d\omega_{K}%
%\]

%\end{corollary}

\medskip To prepare for Remark \ref{AmalgamExtension}(b), we insert the
following definition.

\begin{definition}
\label{Localamalgam}For $p\in\left[  1,\infty\right)  $ we say that a
measurable function $f$ belongs to the amalgam space $\left(  L^{p}%
,\ell^{\infty}\right)  \left(  K\right)  $ if $\left\Vert f\right\Vert
_{p,\infty,U}:=\sup_{x}\left\Vert f\, \left(  \tau_{x}\mathbf{1}_{U}\right)
^{1/p}\right\Vert _{p}$ is finite for some relatively compact neighbourhood
$U$ of the identity.
\end{definition}

In the discussion following Corollary \ref{Amalgam equivalence} below, we show
that replacing $U$ by a different relatively compact neighbourhood of the
identity yields an equivalent norm and hence the same space $(L^{p},
\ell^{\infty})(K)$.
%that $\left \Vert f\right \Vert _{p,\infty,U}$ is then finite for all relatively compact neighbourhoods $U$ of the identity.
Note that
\[
L^{1}\left(  K\right)  \subset(L^{1},l^{\infty})\left(  K\right)  \subset
L_{loc}^{1}\left(  K\right)  .
\]

\begin{remark}
\begin{description}
\item[(a)] \label{AmalgamExtension} In the group case, Corollary \ref{bad}
extends to locally integrable functions $f$ of positive type (see
$\cite{Leinert}$, 1.1 and Theorem 1.6), but for hypergroups this is not always
possible. Indeed the Naimark hypergroup ($\cite{Bloom/Heyer}$, p 99, but note
the misprint in line 5, the second occurrence of $a^{n}$\ should be deleted)
is a counterexample. For this hypergroup on $\mathbb{R}^{+}$\ with Haar
measure $d\omega\left(  x\right)  =\sinh^{2}x\,dx$ there are unbounded
(positive definite) characters of the form $\chi_{a}\left(  x\right)
=\frac{\sinh\left(  rx\right)  }{r\sinh x}$ where $r>1$ and $a=-r^{2}$.
Then$\ \chi_{a}\left(  x\right)  $ behaves like $e^{\left(  r-1\right)  x}$ as
$x\rightarrow\infty$. Writing $U:=\left[  0,1\right]  $, for $x>1$ we have
$0\leq\tau_{x}\mathbf{1}_{U}\leq1$, $\operatorname*{supp}\left(  \tau
_{x}\mathbf{1}_{U}\right)  \subset J_{x}:=\left[  x-1,x+1\right]  $ and
$\int\tau_{x}\mathbf{1}_{U}\,d\omega=\int\mathbf{1}_{U}\,d\omega=:c$, so that
$\tau_{x}\mathbf{1}_{U}\geq\frac{c}{2\omega\left(  J_{x}\right)  }$ on a set
with measure at least $\frac{c}{2}$. Therefore%
\[
\left\Vert \chi_{a}\, \left(  \tau_{x}\mathbf{1}_{U}\right)  ^{1/p}\right\Vert
_{p}\geq\left\Vert \chi_{a}\tau_{x}\mathbf{1}_{U}\right\Vert _{p}\geq\left(
\min_{J_{x}}\chi_{a}\right)  \frac{c}{2\omega\left(  J_{x}\right)  }\left(
\frac{c}{2}\right)  ^{\frac{1}{p}}.
\]
For $a$ sufficiently small ($a<-9$ will do), the right-hand side of this
inequality tends to $\infty$\ as $x\rightarrow\infty$ (and hence
$J_{x}\rightarrow\left\{  \infty\right\}  $), which shows that Corollary
\ref{bad} does not hold on this hypergroup.
\end{description}

\begin{description}
\item[(b)] The proof of Corollary \ref{bad} works for any (locally integrable)
function $f$ of positive type for which the convolutions $f_{\iota}$ all
belong to $L^{\infty}.$ Those convolutions are continuous, of positive type
and (by assumption) bounded, hence positive definite. The $L^{1}$-convergence
in (\ref{Sequence}) can then be replaced by local $L^{1}-$convergence, that is
by convergence in $L^{1}\left(  C\right)  $ for every compact set $C$.

In particular, the proof works for all $f\in\left(  L^{1},\ell^{\infty
}\right)  \left(  K\right)  $ of positive type because the $k_{\iota}$ in our
proof all belong to $C_{c}\left(  K\right)  $. So $f\ast k_{\iota}^{-}\in
L^{\infty},$ as we show in a moment, and hence so does $f_{\iota}=k_{\iota
}\ast f\ast k_{\iota}^{-}$, which shows that $f_{\iota}$ is bounded for each
$\iota$.

For any relatively compact neighbourhood $U\ni e,$ and $\iota$ chosen suitably
large so that $\operatorname*{supp}\left(  k_{\iota}\right)  \subset U$, we
have%
\begin{align*}
\left\vert f\ast k_{\iota}^{-}\right\vert \left(  x\right)   &  \leq
\int\left\vert f\left(  x\ast y\right)  k_{\iota}^{-}\left(  y^{-}\right)
\right\vert d\omega\left(  y\right) \\
&  \leq\left\Vert k_{\iota}\right\Vert _{\infty}\int\left\vert f\left(  x\ast
y\right)  \right\vert \mathbf{1}_{U}\left(  y\right)  d\omega\left(  y\right)
\\
&  \leq\left\Vert k_{\iota}\right\Vert _{\infty}\int\left\vert f\right\vert
\left(  x\ast y\right)  \mathbf{1}_{U}\left(  y\right)  d\omega\left(
y\right) \\
&  =\left\Vert k_{\iota}\right\Vert _{\infty}\int\left\vert f\left(  y\right)
\right\vert \mathbf{1}_{U}\left(  x^{-}\ast y\right)  d\omega\left(  y\right)
\\
&  =\left\Vert k_{\iota}\right\Vert _{\infty}\left\Vert \left\vert
f\right\vert \tau_{x^{-}}\mathbf{1}_{U}\right\Vert _{1}\\
&  \leq\left\Vert k_{\iota}\right\Vert _{\infty}\left\Vert f\right\Vert
_{1,\infty,U}%
\end{align*}
where for the first equality we refer to \cite{Bloom/Heyer}, Theorem 1.3.21,
and hence $f\ast k_{\iota}^{-}$\ is bounded.
\end{description}
\end{remark}

\begin{theorem}
\label{NormCase} Let $K$ be a commutative hypergroup such that $P_{b}(K)\cdot
P_{b}(K)\subset P_{b}(K)$ and let $p\in\mathbb{N}$ be even. For every
relatively compact neighbourhood $U$ of the identity there is a constant
$C_{U}>0$ such that for all choices of $\mu\in M^{1}\left(  K\right)  $ and
$f\in\left(  L^{1},\ell^{\infty}\right)  \left(  K\right)  $
%with $\overset{\wedge}{f}\geq0$ on $K^{\wedge}$\
of positive type we have%
\begin{equation}
\left\Vert f\cdot\left(  \mu\ast\mathbf{1}_{U}\right)  \right\Vert _{p}%
\leq\left\Vert f\cdot\left(  \mu\ast\mathbf{1}_{U}\right)  ^{1/p}\right\Vert
_{p}\leq C_{U}^{1/p}\left\Vert f\, \left(  \mathbf{1}_{U}\right)
^{1/p}\right\Vert _{p}=C_{U}^{1/p}\left\Vert f\, \mathbf{1}_{U}\right\Vert
_{p} \label{powergeneral}%
\end{equation}
In particular this holds for $f\in L^{1}\left(  K\right)  $ of positive type
(equivalently: $f\in L^{1}(K)$ with $\overset{\wedge}{f}\geq0$ on
$\mathrm{supp\,}\pi_{K}$).
\end{theorem}

\begin{proof}
The first inequality in (\ref{powergeneral}) holds for all finite exponents
$p>1$ since $0\leq\mu\ast\mathbf{1}_{U}\leq1.$ The next inequality in
(\ref{powergeneral}) uses Corollary \ref{bad}, the assumption that
$p\in\mathbb{N}$ is even and Remark \ref{AmalgamExtension}(b).
\end{proof}

\begin{corollary}
\label{Infinity norm}Let $K$ be a commutative hypergroup such that
$P_{b}(K)\cdot P_{b}(K)\subset P_{b}(K)$. For $f\in\left(  L^{1},\ell^{\infty
}\right)  \left(  K\right)  $\ of positive type we have%
\begin{equation}
\left\Vert f\right\Vert _{\infty}\leq\left\Vert f\mathbf{1}_{U}\right\Vert
_{\infty}. \label{Infinity norm inequality}%
\end{equation}
In particular, since $0\leq\tau_{x}\mathbf{1}_{U}\leq1$, we have%
\begin{equation}
\left\Vert f\tau_{x}\mathbf{1}_{U}\right\Vert _{\infty}\leq\left\Vert
f\mathbf{1}_{U_{x}}\right\Vert _{\infty}\leq\left\Vert f\mathbf{1}%
_{U}\right\Vert _{\infty} \label{Second infinity norm inequality}%
\end{equation}
where $U_{x}=\left\{  y\left\vert \tau_{x}\mathbf{1}_{U}(y)>0\right.
\right\}  $.
\end{corollary}

\begin{proof}
The second quantity in (\ref{powergeneral}) is the $L^{p}$ norm of $f$
relative to the measure $\left(  \mu\ast\mathbf{1}_{U}\right)  d\omega$. Since
the total mass of this measure is finite, letting $p\rightarrow\infty$ in
(\ref{powergeneral})\ gives the essential supremum of $\left\vert f\right\vert
$\ on the set where $\mu\ast\mathbf{1}_{U}>0$. Apply this with $\mu
=\varepsilon_{x}$\ for various points $x$ in $K,$ and use the fact that
$U_{x}$ is a neighbourhood of $x^{-}$, to obtain $\left\Vert f\right\Vert
_{\infty}\leq\left\Vert f\mathbf{1}_{U}\right\Vert _{\infty}$.
\end{proof}

\begin{remark}
\label{translation} Note that taking $\mu=\varepsilon_{x}$ in Theorem
\ref{NormCase} gives that for all even $p\in\mathbb{N}$%
\begin{equation}
\left\Vert f\, \tau_{x}\mathbf{1}_{U}\right\Vert _{p}\leq\left\Vert f\left(
\tau_{x}\mathbf{1}_{U}\right)  ^{1/p}\right\Vert _{p}\leq C_{U}^{1/p}%
\left\Vert f\, \mathbf{1}_{U}\right\Vert _{p}. \label{power}%
\end{equation}
It is useful to recall at this stage that for fixed $p$, the quantities
$\left\Vert f\, \tau_{x}\mathbf{1}_{U}\right\Vert _{p}$ and $\left\Vert f\,
\left(  \tau_{x}\mathbf{1}_{U}\right)  ^{1/p}\right\Vert _{p}$ agree on groups
but not necessarily on hypergroups (see the end of Remark \ref{remark20} below).
\end{remark}

We restate (\ref{Second infinity norm inequality}) and (\ref{power}) using
Definition \ref{Localamalgam}.

\begin{corollary}
\label{Amalgam equivalence}(Wiener's theorem for functions in $\left(
L^{1},\ell^{\infty}\right)  \left(  K\right)  $) Let $K$ be a commutative
hypergroup such that $P_{b}(K)\cdot P_{b}(K)\subset P_{b}(K)$ and take
$p\in\mathbb{N}$ even or $p=\infty.$ If $f\in\left(  L^{1},\ell^{\infty
}\right)  \left(  K\right)  $ is of positive type,
%with $\overset{\wedge}{f}\geq0$ on $K^{\wedge}$
and satisfies $f\, \mathbf{1}_{U}\in L^{p}\left(  K\right)  $ for some
relatively compact neighbourhood $U$\ of $e$, then
\[
f\in\left(  L^{p},\ell^{\infty}\right)  \left(  K\right)  \quad\text{and}%
\quad\|f\|_{p, \infty, U} \leq C^{1/p}_{U}\left\Vert f\, \mathbf{1}%
_{U}\right\Vert _{p}.
\]
%$f\in \left(  L^{p}%
%,\ell^{\infty}\right)  \left(  K\right) $ and~$\|f\|_{p, \infty, U}  \leq C^{1/p}_{U}%
%\left \Vert f\, \mathbf{1}_{U}\right \Vert _{p}$.
In particular this holds for $f\in L^{1}\left(  K\right)  $ satisfying the
same conditions.
\end{corollary}

Note that, by the equivalence proved next, if $K$ is compact, then
$(L^{p},\ell^{\infty})=L^{p}$ and $\Vert\cdot\Vert_{p,\infty,U}$ equals (up to
equivalence) the $L^{p}$ norm on $K$ (take $\Vert\cdot\Vert_{p,\infty,K}$ and
use $\tau_{x}\mathbf{1}_{K}=\mathbf{1}_{K})$.
%On abelian groups, these amalgam spaces
%are known \cite{Feichtinger}, \cite{FournierStewart}
%to have various good properties.
%%In later sections of this paper
%we will discuss these matters for two interesting
%abelian  hypergroups.
%We now show that even on nonabelian hypergroups,
%these spaces do not depend on the choice of $U$.
%Let~$V$ be another relatively compact neighbourhood of $e$
%in the same hypergroup $K$.
%Denote the amalgam spaces corresponding to $U$ and $V$
%by $(L^p, \ell^\infty)_U$
%and $(L^p, \ell^\infty)_V$.
%There are $\lambda_i > 0$ and $x_i$ in $K$ such that
%${\mathbf 1}_U \le \sum_{i=1}^r \lambda_i\tau_{x_i}{\mathbf 1}_V$.
%For $f \in (L^p, \ell^\infty)_V$ and $x \in K$, we obtain
%\[
%\left\|  f\,\tau_{x}\mathbf{1}_{U}\right\|  _{p}
%\le \sum_{i=1}^r \lambda_i
%\left\|  f\,\tau_{x}\tau_{x_i}\mathbf{1}_{V}\right\|  _{p}
%\]

%Proof that $(L^p,l^\infty)$ does not depend
%on the neighbourhood used:

We now compare $\Vert f\Vert_{p,\infty,U}$ for different choices of $U$ (even
on non-commutative hypergroups). Let $U$ and $V$ be relatively compact
neighbourhoods of $e$, and denote the corresponding amalgam spaces by
$(L^{p},\ell^{\infty})_{U}$ and $(L^{p},\ell^{\infty})_{V}$ respectively.
There are $\lambda_{i}>0$ and $x_{i}\in K$ such that $1_{U}\leq\sum_{i=1}%
^{n}\lambda_{i}\tau_{x_{i}}\mathbf{1}_{V}$. Let $f\in(L^{p},\ell^{\infty}%
)_{V}$ and $x\in K.$ When $1\leq p<\infty$ we have
\[
\left\Vert f\, \left(  \tau_{x}1_{U}\right)  ^{1/p}\right\Vert _{p}^{p}%
=\int\left\vert f\right\vert ^{p}\tau_{x}1_{U}\,d\omega_{K}\leq\int\left\vert
f\right\vert ^{p}\tau_{x}\left(  \sum_{i=1}^{n}\lambda_{i}\tau_{x_{i}%
}\mathbf{1}_{V}\right)  \,d\omega_{K}%
\]%
\[
=\sum_{i=1}^{n}\lambda_{i}\left\Vert f\, \left(  \tau_{x}\tau_{x_{i}%
}\mathbf{1}_{V}\right)  ^{1/p}\right\Vert _{p}^{p}\leq\left(  \sum_{i=1}%
^{n}\lambda_{i}\right)  \Vert f\Vert_{p,\infty,V}^{p}%
\]
by Lemma \ref{Amalgam}\ below (set $\mu=\varepsilon_{x^{-}}\ast\varepsilon
_{x_{i}^{-}}$). Hence
\[
f\in(L^{p},\ell^{\infty})_{U}\quad\text{and}\quad\Vert f\Vert_{p,\infty,U}\leq
C\Vert f\Vert_{p,\infty,V}%
\]
with $C=\left(  \sum_{i=1}^{n}\lambda_{i}\right)  ^{1/p}$, so that the amalgam
space $\left(  L^{p},\ell^{\infty}\right)  \left(  K\right)  $ does not depend
on the chosen neighbourhood.

Note that, since necessarily $\sum\lambda_{i} \geq1$, this sum can serve as a
constant for all finite $p$. So we have constants of equivalence which only
depend on $U$ and $V$, but not on $p$.

If $p=\infty$ and (as before) we denote by $U_{x}$\ the set where $\tau
_{x}\mathbf{1}_{U}>0$, then $\left\Vert f\right\Vert _{\infty,\infty,U}%
=\sup_{x}\left\Vert f\mathbf{1}_{U_{x}}\right\Vert _{\infty}$. Since $U_{x}%
$\ is a neighbourhood of $x$, we obtain $\left\Vert f\right\Vert
_{\infty,\infty,U}= \left\Vert f\right\Vert _{\infty}$. So in this case, if we
use $V$ instead of $U$, we obtain not only an equivalent norm but in fact the
very same norm.

\begin{lemma}
\label{Amalgam}Let $p\in\left[  1,\infty\right]  .$ For $f\in(L^{p}%
,\ell^{\infty})_{V}$ and $\mu$ a probability measure with compact support we
have $f\left(  \mu\ast\mathbf{1}_{V}\right)  ^{1/p}\in L^{p}$ and $\Vert
f\left(  \mu\ast\mathbf{1}_{V}\right)  ^{1/p}\Vert_{p}\leq\Vert f\Vert
_{p,\infty,V}$.
\end{lemma}

\begin{proof}
By \cite{LeinBook} Proposition 13.64 and the remarks following it, the set $S$
of all convex linear combinations of Dirac measures is weakly dense in
$M^{1}(K)$. So there is a net $(\mu_{\iota})$ in $S$ with $\mu_{\iota
}\rightarrow\mu$ weakly. In the present case we may assume $\mathrm{supp\,}%
\mu_{\iota}\subset\mathrm{supp\,}\mu$ (in the proof of \cite{LeinBook}, 13.64,
if $A_{j}\, \cap\mathrm{supp\,}\mu\neq\emptyset$, choose $x_{j}$ in this set
and not just in $A_{j}$). By \cite{Bloom/Heyer} Theorem 1.6.18(b) we obtain
$\Vert\mu_{\iota}\ast g-\mu\ast g\Vert_{1}\rightarrow0$ for all $g \in
L^{1}(K)$. From the net $\left(  \mu_{\iota}\ast\mathbf{1}_{V}\right)  $ we
may extract a sequence $\left(  \mu_{n}\ast\mathbf{1}_{V}\right)  $ converging
in $\Vert\cdot\Vert_{1}$ and (if necessary, passing to a subsequence thereof)
also pointwise $a.e.$\ to $\mu\ast\mathbf{1}_{V}$. Hence
\[
\left(  \mu_{n}\ast\mathbf{1}_{V}\right)  ^{1/p}\rightarrow\left(  \mu
\ast\mathbf{1}_{V}\right)  ^{1/p}\; \text{\textit{a.e.}}%
\]
All these functions have absolute value $\leq1$ (see \cite{Bloom/Heyer},
1.4.6) and have support in the compact set $\mathrm{supp\,}(\mu)\ast
\mathrm{supp\,}(\mathbf{1}_{V})$ (see \cite{Bloom/Heyer}), 1.2.12), hence are
dominated by $h=\mathbf{1}_{\mathrm{supp\,}(\mu)\ast\mathrm{supp\,}%
(\mathbf{1}_{V})}$. There are $\beta_{k}>0$ and $y_{k}\in K$ such that
$h\leq\sum_{k=1}^{l}\beta_{k}\left(  \tau_{y_{k}}\mathbf{1}_{V}\right)
^{1/p}$, so
\[
\Vert f\,h\Vert_{p}\leq\sum_{k=1}^{l}\beta_{k}\Vert f\left(  \, \tau_{y_{k}%
}\mathbf{1}_{V}\right)  ^{1/p}\Vert_{p}<\infty.
\]
By dominated convergence we obtain $\Vert f(\mu_{n}\ast\mathbf{1}_{V}%
)^{1/p}-f(\mu\ast\mathbf{1}_{V})^{1/p}\Vert_{p}\rightarrow0$. Now, since
$\mu_{n}$ is a convex combination $\sum_{j=1}^{m}\gamma_{j}\varepsilon_{x_{j}%
}$, we have
\[
\Vert f\left(  \mu_{n}\ast\mathbf{1}_{V}\right)  ^{1/p}\Vert_{p}^{p}=\Vert
f\left(  \sum_{j=1}^{m}\gamma_{j}\tau_{{x_{j}^{-}}}\mathbf{1}_{V}\right)
^{1/p}\Vert_{p}^{p}=\int\left\vert f\right\vert ^{p}\sum_{j=1}^{m}\gamma
_{j}\tau_{{x_{j}^{-}}}\mathbf{1}_{V}\,d\omega_{K}%
\]%
\[
=\sum_{j=1}^{m}\gamma_{j}\Vert f\, \left(  \tau_{x_{j}^{-}}\mathbf{1}%
_{V}\right)  ^{\frac{1}{p}}\Vert_{p}^{p}\leq\sum_{j=1}^{m}\gamma_{j}\Vert
f\Vert_{p,\infty,V}^{p}=\Vert f\Vert_{p,\infty,V}^{p}.
\]
Hence $\Vert f\left(  \mu\ast\mathbf{1}_{V}\right)  ^{1/p}\Vert_{p}\leq\Vert
f\Vert_{p,\infty,V}$ as asserted.
\end{proof}

%\begin{remark}
%All the above results continue to hold for those hypergroups for which the
%product of two bounded continuous positive definite functions is again
%positive definite.

\begin{remark}
All of the results obtained so far hold for a large class of commutative
hypergroups, in particular for strong hypergroups, and hence also for those
examples to be considered below. Furthermore, much of this section extends to
some non-commutative hypergroups. A version of Lemma \ref{Property*} holds
without the assumption that $K$ is commutative. Instead, we assume that there
is a relatively compact neighbourhood $V$ of the identity with the property
that $\mathbf{1}_{V}$ is central in the convolution algebra $L^{1}(K)$ and
hence in the measure algebra on $K$. The conclusion of the lemma then holds
for neighbourhoods $U$ of $e$ that include the support of the product
$\mathbf{1}_{V}^{\sim}\ast\mathbf{1}_{V}$.
%We note that in the cases where~$K$ is not unimodular,
%$\widetilde{\mathbf1_{V}}(x)$ should be interpreted
%as~$\Delta(x^-)
%\mathbf1_{V}(x)$.
The centrality assumption
%on $\mathbf 1_{V}$
implies that $K$ is unimodular. In particular, $(\mathbf{1}_{V})^{\ast
}=\mathbf{1}_{V}^{\sim}$ (as in the commutative case). Therefore the proof of
the lemma remains almost the same
%But instead of choosing~$V$ given any $U$,
(replace the sentence concerning the supports of the $\varepsilon_{x^{-}}%
\ast\varepsilon_{y}$ up to and including inequality (\ref{nesting}) by "Let
$h=\mathbf{1}_{V}^{\sim}\ast\mathbf{1}_{V}$."). With the same modified
hypothesis, Corollary \ref{special case} holds with no change in its proof.
For Corollary \ref{bad}
%one may have to take ~$U$ to be slightly bigger (such
we also require that the support of $\mathbf{1}_{V}^{\sim}\ast\mathbf{1}_{V}$
be contained in the interior of $U$, rather than just in $U$. In the proof of
Corollary \ref{bad} take $W$ equal to this support. Then for such
$U$,\ Theorem \ref{NormCase} and hence Remark \ref{translation} as well as
Corollary \ref{Amalgam equivalence} for even $p$ also hold. For $p=\infty$,
Corollary \ref{Infinity norm} and hence the corresponding part of Corollary
\ref{Amalgam equivalence} hold on general hypergroups (without any centrality
assumption):
%\end{remark}

Let $f\in\left(  L^{1},\ell^{\infty}\right)  \left(  K\right)  $ be of
positive type. If $U$ is a relatively compact neighbourhood of $e$ and
$f_{\iota}=k_{\iota}\ast f\ast k_{\iota}^{-}$ where the $k_{\iota}$\ are as in
the proof of Corollary \ref{bad}, take $\iota$ large enough so that
$\operatorname*{supp}\left(  k_{\iota}^{\ast}\ast k_{\iota}\right)  \subset
U$. Then (see Remark \ref{AmalgamExtension}) $f_{\iota}$ is continuous,
positive definite and bounded, so by $\cite{Bloom/Heyer}$, Lemma 4.1.3(g) for
the first equality and (\ref{Left-hand version}) for the third equality below,
we have%
\begin{align*}
\left\Vert f_{\iota}\right\Vert _{\infty}  &  =k_{\iota}\ast f\ast k_{\iota
}^{-}\left(  e\right) \\
&  =\int\left(  k_{\iota}\ast f\right)  k_{\iota}\,d\omega_{K}\\
&  =\int f\cdot\left(  k_{\iota}^{\ast}\ast k_{\iota}\right)  d\omega_{K}\\
&  \leq\left\Vert f\mathbf{1}_{U}\right\Vert _{\infty}\left\Vert k_{\iota
}^{\ast}\ast k_{\iota}\right\Vert _{1}\\
&  \leq\left\Vert f\mathbf{1}_{U}\right\Vert _{\infty}.
\end{align*}
Since $f_{\iota}\rightarrow f$ locally in $L^{1}-$norm (that is, $\left\Vert
\left(  f_{\iota}-f\right)  \mathbf{1}_{C}\right\Vert _{1}\rightarrow0$ for
every compact $C\subset K$), we obtain $\left\Vert f\right\Vert _{\infty}%
\leq\left\Vert f\mathbf{1}_{U}\right\Vert _{\infty}$.
\end{remark}

\section{Hypergroups on $\mathbb{R}_{+}\label{hypergroups}$}

In this section we consider some hypergroups on $\mathbb{R}_{+}$ to which all
of Section 2 applies. For one of them we show that the version of Wiener's
theorem presented in \cite{Fournier1} for locally compact abelian groups also
holds (Theorem \ref{transforms} below ),
%{SquareIntegrable})%
%),
as indeed do other positive results about translation, convolution and Fourier
transforms, which we need for the proof of the theorem.

\subsection{Bessel-Kingman hypergroups}

For these hypergroups the reader is referred to \cite{Bloom/Heyer}, Section
3.5.61, but we give here some basic properties. Let $\alpha>-\frac{1}{2}%
.$\ For $x,y\in\mathbb{R}_{+}$\ consider the convolution%
\[
\varepsilon_{x}\ast_{\alpha}\varepsilon_{0}=\varepsilon_{x}=\varepsilon
_{0}\ast_{\alpha}\varepsilon_{x}%
\]
and for $x,y>0,$
\[
\varepsilon_{x}\ast_{\alpha}\varepsilon_{y}\left(  f\right)  =\int_{\left\vert
x-y\right\vert }^{x+y}K_{\alpha}\left(  x,y,z\right)  f\left(  z\right)
z^{2\alpha+1}\,dz,\,f\in C_{0}\left(  \mathbb{R}_{+}\right)
\]
where%
\[
K_{\alpha}\left(  x,y,z\right)  :=\frac{\Gamma\left(  \alpha+1\right)
}{\Gamma\left(  \frac{1}{2}\right)  \Gamma\left(  \alpha+\frac{1}{2}\right)
2^{2\alpha-1}}\frac{\left[  \left(  z^{2}-\left(  x-y\right)  ^{2}\right)
\left(  \left(  x+y\right)  ^{2}-z^{2}\right)  \right]  ^{\alpha-\frac{1}{2}}%
}{\left(  xyz\right)  ^{2\alpha}}.
\]
Then $\left(  \mathbb{R}_{+},\ast_{\alpha}\right)  $\ is a commutative
hypergroup with the identity involution and Haar measure $\omega_{\alpha
}\left(  dz\right)  =z^{2\alpha+1}dz.$\ Its characters are given by
$\varphi_{\lambda}\left(  x\right)  :=j_{\alpha}\left(  \lambda x\right)  $,
$x\in\mathbb{R}_{+}$\ for each $\lambda\geq0$\ where $j_{\alpha}$\ denotes the
modified Bessel function of order $\alpha$\ given by
\[
j_{\alpha}\left(  x\right)  :=%
%TCIMACRO{\dsum _{k=0}^{\infty}}%
%BeginExpansion
{\displaystyle\sum_{k=0}^{\infty}}
%EndExpansion
\frac{\left(  -1\right)  ^{k}\Gamma\left(  \alpha+1\right)  }{2^{2k}%
k!\Gamma\left(  \alpha+k+1\right)  }x^{2k},\text{ }x\in\mathbb{R}.
\]
Note that $\varphi_{0}\equiv1.$\ 

It is well known that $\left(  \mathbb{R}_{+},\ast_{\alpha}\right)
\cong\left(  \mathbb{R}_{+},\ast_{\alpha}\right)  ^{\wedge},$ where the
hypergroup isomorphism is given by $\lambda\longmapsto\varphi_{\lambda}$ (so
that $\left(  \mathbb{R}_{+},\ast_{\alpha}\right)  \ $is not only strong but
even Pontryagin); see \cite{Zeuner}, Example 7.2. Wiener's theorem as in
Corollary \ref{Amalgam equivalence} therefore holds for these Bessel-Kingman hypergroups.

For $\alpha=\frac{1}{2}$\ (the motion hypergroup) the convolution is given by
\begin{equation}
\varepsilon_{x}\ast_{\frac{1}{2}}\varepsilon_{y}\left(  f\right)  =\frac
{1}{2xy}\int_{\left\vert x-y\right\vert }^{x+y}f\left(  z\right)  z\,dz
\label{MotionConvolution}%
\end{equation}
in which case the characters are just
\[
\varphi_{\lambda}\left(  x\right)  =j_{\frac{1}{2}}\left(  \lambda x\right)  =%
%TCIMACRO{\dsum _{k=0}^{\infty}}%
%BeginExpansion
{\displaystyle\sum_{k=0}^{\infty}}
%EndExpansion
\frac{\left(  -1\right)  ^{k}\Gamma\left(  \frac{3}{2}\right)  }%
{2^{2k}k!\Gamma\left(  k+\frac{3}{2}\right)  }\left(  \lambda x\right)
^{2k}=\frac{\sin\lambda x}{\lambda x},\lambda\geq0.
\]
The term `motion hypergroup' is justified by the fact that $\left(
\mathbb{R}_{+},\ast_{\frac{1}{2}}\right)  $ is isomorphic to the double coset
space $M\! \left(  3\right)
%TCIMACRO{\TeXButton{dcoset}{\dcoset}}%
%BeginExpansion
\dcoset
%EndExpansion
SO\! \left(  3\right)  .$

For $f\in L^{1}\left(  \mathbb{R}_{+},\ast_{\alpha}\right)  ,\alpha>-\frac
{1}{2},$ its Fourier transform is defined by
\[
\overset{\wedge}{f}\left(  \varphi_{\lambda}\right)  :=\int_{\mathbb{R}_{+}%
}f\, \varphi_{\lambda}\,d\omega_{\alpha}%
\]
and the convolution of two functions $f,g$\ is given by
\[
f\ast_{\alpha}g\left(  x\right)  :=\int_{\mathbb{R}_{+}}f\left(  x\ast
_{\alpha}y\right)  g\left(  y\right)  \, \omega_{\alpha}\left(  dy\right)  .
\]
Recall that
\[
\left(  f\ast_{\alpha}g\right)  ^{\wedge}=\overset{\wedge}{f}\overset{\wedge
}{g}.
\]

When $\alpha=\frac{1}{2}$ we have%
\[
\overset{\wedge}{f}\left(  \varphi_{\lambda}\right)  =\left\{
\begin{array}
[c]{ll}%
\dfrac{1}{\lambda}%
%TCIMACRO{\dint _{\mathbb{R}_{+}}}%
%BeginExpansion
{\displaystyle\int_{\mathbb{R}_{+}}}
%EndExpansion
f\left(  x\right)  \left(  \sin\lambda x\right)  x\,dx, & \lambda\neq0,\\
& \\
\int_{\mathbb{R}_{+}}f\left(  x\right)  x^{2}\,dx, & \lambda=0.
\end{array}
\right.
\]
and, in particular,
\begin{equation}
\left(  \mathbf{1}_{\left[  0,\varepsilon\right)  }\right)  ^{\wedge}\left(
\varphi_{\lambda}\right)  =\left\{
\begin{array}
[c]{ll}%
\dfrac{1}{\lambda^{3}}\left(  \sin\lambda\varepsilon-\lambda\varepsilon
\cos\lambda\varepsilon\right)  , & \lambda\neq0,\\
& \\
\dfrac{\varepsilon^{3}}{3}, & \lambda=0.
\end{array}
\right.  \label{IndicatorFourierTransform}%
\end{equation}

\subsection{The amalgam spaces $\left(  L^{p},\ell^{q}\right)  \left(
\mathbb{R}_{+},\ast_{\alpha}\right)  $ for\newline$1\leq p,q\leq\infty$}

In preparation for Theorem \ref{transforms} in Section \ref{fournier}, we need
to develop some properties of certain \textit{discrete} amalgam spaces.
%We do that in this section and the next one.
We define them so that the norms $\left\Vert \cdot\right\Vert _{p,\infty}%
$\ used in this section are equivalent to the corresponding
\textit{continuous} norms $\Vert\cdot\Vert_{p,\infty,U}$ used in Section
\ref{Wiener}, and we prove this equivalence in Section \ref{equisection}. At
the end of the current subsection, we consider other families of discrete
amalgam norms, in particular those introduced in \cite{CowlingMedaPasquale},
and show that they are mostly not equivalent to the norms that we use.

For each $n\in\mathbb{N}$ write $I_{n}:=\left[  n-1,n\right)  $ and for $1\leq
p,q\leq\infty$ define
\begin{equation}
\left\Vert f\right\Vert _{p,q}:=\left(  \sum_{n=1}^{\infty}\omega_{\alpha
}\left(  I_{n}\right)  \left(  \frac{1}{\omega_{\alpha}\left(  I_{n}\right)
}\int_{n-1}^{n}\left\vert f\right\vert ^{p}\,d\omega_{\alpha}\right)
^{\frac{q}{p}}\right)  ^{\frac{1}{q}} \label{DiscreteNorm}%
\end{equation}
with the usual convention when one or both of $p,q$\ is $\infty$, that is%
\begin{align*}
\left\Vert f\right\Vert _{\infty,q}  &  =\left(  \sum_{n=1}^{\infty}%
\omega_{\alpha}\left(  I_{n}\right)  \sup_{x\in I_{n}}\left\vert f\left(
x\right)  \right\vert ^{q}\right)  ^{\frac{1}{q}}\\
\left\Vert f\right\Vert _{p,\infty}  &  =\sup_{n}\left(  \frac{1}%
{\omega_{\alpha}\left(  I_{n}\right)  }\int_{n-1}^{n}\left\vert f\right\vert
^{p}\,d\omega_{\alpha}\right)  ^{\frac{1}{p}}%
\end{align*}
and
\[
\left\Vert f\right\Vert _{\infty,\infty}=\sup_{n}\sup_{x\in I_{n}}\left\vert
f\left(  x\right)  \right\vert =\left\Vert f\right\Vert _{\infty}.
\]
The $\left(  p,q\right)  -$amalgam space is defined as the subspace of all
measurable functions $f$\ given by
\[
\left(  L^{p},\ell^{q}\right)  \left(  \mathbb{R}_{+},\ast_{\alpha}\right)
=\left\{  f:\left\Vert f\right\Vert _{p,q}<\infty\right\}  .
\]
We have the following result.

\begin{proposition}
\label{Inclusion}Let $f$\ be a measurable function. Then for $p_{1}\leq p_{2}%
$\
\[
\left\Vert f\right\Vert _{p_{1},q}\leq\left\Vert f\right\Vert _{p_{2},q}%
\]
and for $q_{1}\geq q_{2}$%
\[
\left\Vert f\right\Vert _{p,q_{1}}\leq C\left\Vert f\right\Vert _{p,q_{2}}%
\]
where $C$\ is a constant. In particular, for $p_{1}\leq p_{2}$ and $q_{1}\geq
q_{2}$%
\[
\left(  L^{p_{2}},\ell^{q_{2}}\right)  \left(  \mathbb{R}_{+},\ast_{\alpha
}\right)  \subset\left(  L^{p_{1}},\ell^{q_{1}}\right)  \left(  \mathbb{R}%
_{+},\ast_{\alpha}\right)
\]
and
\begin{align*}
&  \left(  L^{p},\ell^{q}\right)  \left(  \mathbb{R}_{+},\ast_{\alpha}\right)
\subset L^{p}\left(  \mathbb{R}_{+},\ast_{\alpha}\right)  \cap L^{q}\left(
\mathbb{R}_{+},\ast_{\alpha}\right)  \text{ for }p\geq q,\\
& \\
&  L^{p}\left(  \mathbb{R}_{+},\ast_{\alpha}\right)  \cup L^{q}\left(
\mathbb{R}_{+},\ast_{\alpha}\right)  \subset\left(  L^{p},\ell^{q}\right)
\left(  \mathbb{R}_{+},\ast_{\alpha}\right)  \text{ for }p\leq q.
\end{align*}

\end{proposition}

\begin{proof}
This is straightforward using H\"{o}lder's inequality together with the
property that $\omega_{\alpha}\left(  I_{n}\right)  \geq C>0$ for all
$n.\vspace{0.25cm}$
\end{proof}

Note that $\left(  L^{\infty},\ell^{1}\right)  \left(  \mathbb{R}_{+}%
,\ast_{\alpha}\right)  $ is the smallest amalgam space and $\left(  L^{1}%
,\ell^{\infty}\right)  \left(  \mathbb{R}_{+},\ast_{\alpha}\right)  $ is the largest.

\begin{remark}
We now use indicator functions on subintervals of $I_{n}$ to show that for
$p\neq q$ our amalgam norms are not equivalent to the discrete amalgam norms
in $\cite{CowlingMedaPasquale}$, which are computed on sets with measures
uniformly bounded away from $0$ and $\infty$. There is no division or
multiplication by measures of tiles in the computation of those norms. In the
present case we obtain norms equivalent to those in
$\cite{CowlingMedaPasquale}$ by splitting $\mathbb{R}_{+}$\ into disjoint
intervals of Haar measure $1$; at least $\omega_{\alpha}\left(  I_{n}\right)
-2$ of these subintervals are included in $I_{n}$. Let $f$ be the indicator
function of one such subinterval. Then the norm of $f$ in our space $\left(
L^{p},\ell^{q}\right)  \left(  \mathbb{R}_{+},\ast_{\alpha}\right)  $\ is
$\omega_{\alpha}\left(  I_{n}\right)  ^{1/q-1/p}$, while its norm in the
corresponding space in $\cite{CowlingMedaPasquale}$ is $1$. Since
$\omega_{\alpha}\left(  I_{n}\right)  \rightarrow\infty$ as $n\rightarrow
\infty$, these norms are not equivalent unless $p=q$.

Both families of discrete amalgams on Bessel-Kingman hypergroups are
constructed in such a way as to have norms equal to the usual $L^{p}$ norm,
and hence to each other when $p=q$. In the
%The
examples above, the
%show that our norms $\left \Vert\cdot \right \Vert _{p,q}$ are otherwise not equivalent to those in \cite{CowlingMedaPasquale}. The above
functions $f$ are not positive definite, and we do not know whether there are
corresponding examples involving positive definite functions. Finally, most
other
%reasonable versions of
choices give amalgam norms that are not
%also fail to give norms
equivalent to
%ours.
ours,
%Two such examples are the
for example the
%obvious amalgam
partition choice having the $I_{n}$ without normalization, and the
%a
continuous amalgam norm as in Definition \ref{Localamalgam}
%,
but without the $1/p$ power.
%power,
%are
%is
%inequivalent to all the norms just mentioned, even in the case $p=q$.
The only cases where our discrete amalgam norm is equivalent to the one
without weights are those where $p=q$, and the only case
%cases
where the two kinds of continuous amalgam norms are equivalent is that where
%are those
$p=1$ (see the end of Remark \ref{remark20}).
\end{remark}

\subsection{Equivalence of the discrete amalgam norm $\left\Vert
\cdot\right\Vert _{p,\infty}$ with the continuous\ amalgam norm defined by
translations in the case $\alpha=\frac{1}{2}$}

\label{equisection}For the following subsections of Section \ref{hypergroups}%
\ we only consider the Bessel-Kingman hypergroup\ $\left(  \mathbb{R}_{+}%
,\ast_{\frac{1}{2}}\right)  $\ (and to simplify the notation we write $\omega$
in place of $\omega_{\frac{1}{2}}$). Values of $\alpha>\frac{1}{2}$\ are
%have been
treated in \cite{Innig}.

\begin{proposition}
\label{Equivalence1}For $p\in\left[  1,\infty\right)  ,$%
\[
\left\Vert f\right\Vert _{p,\infty}\leq C\sup_{y\in\left[  0,\infty\right)
}\left(  \int\left\vert f\right\vert ^{p}\tau_{y}\mathbf{1}_{\left[
0,1\right]  }\,d\omega\right)  ^{1/p}.
\]

\end{proposition}

\begin{proof}
We have using (\ref{MotionConvolution})%
\begin{align}
\tau_{y}\mathbf{1}_{\left[  0,1\right]  }\left(  x\right)   &  =\mathbf{1}%
_{\left[  0,1\right]  }\left(  y\ast_{\frac{1}{2}}x\right) \nonumber\\
&  =\varepsilon_{y}\ast_{\frac{1}{2}}\varepsilon_{x}\left(  \mathbf{1}%
_{\left[  0,1\right]  }\right) \nonumber\\
&  =\frac{1}{2xy}\int_{\left[  \left\vert x-y\right\vert ,x+y\right]
\cap\left[  0,1\right]  }t\,dt\nonumber\\
&  =\left\{
\begin{array}
[c]{cc}%
1, & x+y\leq1,\\
& \\
\frac{1}{4xy}\left(  1-\left(  x-y\right)  ^{2}\right)  , & x+y>1\text{ and
}\left\vert x-y\right\vert <1,\\
& \\
0, & \left\vert x-y\right\vert \geq1.
\end{array}
\right.  \label{Three cases}%
\end{align}
For $y=n+\frac{1}{2},n\in\mathbb{N},$ we obtain%
\[
\tau_{n+\frac{1}{2}}\mathbf{1}_{\left[  0,1\right]  }\left(  x\right)
=\left\{
\begin{array}
[c]{cc}%
\frac{1-\left(  n+\frac{1}{2}-x\right)  ^{2}}{4x\left(  n+\frac{1}{2}\right)
}, & \left\vert n+\frac{1}{2}-x\right\vert <1,\\
& \\
0, & \left\vert n+\frac{1}{2}-x\right\vert \geq1.
\end{array}
\right.
\]
On the interval $I_{n+1}$ this is larger than%
\[
\frac{3/4}{4\left(  n+\frac{1}{2}\right)  \left(  n+1\right)  }\geq\frac
{3/16}{2\omega\left(  I_{n+1}\right)  }%
\]
which holds for all $n\in\mathbb{N}$. On $I_{1}$ we have the trivial estimate
$\tau_{0}\mathbf{1}_{\left[  0,1\right]  }\geq1=\frac{1}{3}\frac{1}%
{\omega\left(  I_{1}\right)  }$, and putting these together gives%
\[
\sup_{y\in\left[  0,\infty\right)  }\left(  \int\left\vert f\right\vert
^{p}\tau_{y}\mathbf{1}_{\left[  0,1\right]  }\,d\omega\right)  ^{1/p}\geq
\frac{3}{32}\sup_{n\in\mathbb{N}}\left(  \int_{I_{n}}\frac{1}{\omega\left(
I_{n}\right)  }\left\vert f\right\vert ^{p}\,d\omega\right)  ^{1/p}.
\]

\hspace{12cm}
\end{proof}

\begin{remark}
\label{remark20} In Proposition \ref{Equivalence1} we compared the norm
$\left\Vert f\right\Vert _{p,\infty}$\ with the continuous amalgam norm
%and\newline
$\|f\|_{p, \infty, [0,1]} = \sup_{y\in\left[  0,\infty\right)  }\left\Vert
f\left(  \tau_{y}\mathbf{1}_{\left[  0,1\right]  }\right)  ^{\frac{1}{p}%
}\right\Vert _{p}$ for $p\in\left[  1,\infty\right)  $. We consider the same
comparison with $p=\infty.$ Letting%
\[
A\left(  y\right)  =\left\{  t\in\left[  0,\infty\right)  :\tau_{y}%
\mathbf{1}_{\left[  0,1\right]  }\left(  t\right)  >0\right\}
\]
we have $\|f\|_{\infty,\infty, [0,1]} = \sup_{y \in[0, \infty)}\left\|
f\mathbf{1}_{A\left(  y\right)  }\right\|  _{\infty}$.
%note that the measures $\mathbf{1}_{A\left(  y\right)  }\,\omega$ and $\tau_{y}\mathbf{1}_{\left[  0,1\right]  }\omega$\ are equivalent in the sense that they have the same null sets, and so the essential supremum norms $\left\Vert f\right\Vert _{\infty,\mathbf{1}_{A\left(  y\right)  }\,\omega}$ and $\left\Vert f\right\Vert _{\infty,\tau_{y}\mathbf{1}_{\left[  0,1\right] }\,\omega}$ with respect to them are equal.
%\[
%\left\Vert f\right\Vert _{\infty,\mathbf{1}_{A\left(  y\right)  }\,\omega}=\left\Vert f\right\Vert _{\infty,\tau_{y}\mathbf{1}_{\left[  0,1\right]}\,\omega}.
%\]
Clearly $A\left(  y\right)  $\ is an open neighbourhood of $y$\ and hence%
\[
\sup_{y \in[0, \infty)}\left\|  f\mathbf{1}_{A\left(  y\right)  }\right\|
_{\infty}
%\sup_{y\in\left[  0,\infty\right)  }\left\Vert f\right\Vert _{\infty,\tau_{y}\mathbf{1}_{\left[  0,1\right]  }\,\omega}
%=\sup_{y\in\left[0,\infty\right)  }\left\Vert f\right\Vert _{\infty,\mathbf{1}_{A\left(y\right)  }\,\omega}
=\left\Vert f\right\Vert _{\infty}=\left\Vert f\right\Vert _{\infty,\infty}.
\]
This means that for $p=\infty$\ we have $C=1$ and in fact equality in
Proposition \ref{Equivalence1}.

We warn the reader that for every $p\in(1,\infty]$ the seemingly similar (and,
in the group case, identical) norm $\sup_{y\in\left[  0,\infty\right)
}\left\Vert f
%\left(
\tau_{y}\mathbf{1}_{\left[  0,1\right]  }
%\right)
%^{\frac{1}{p}}
\right\Vert _{p}$ is smaller and not equivalent to $\sup_{y\in\left[
0,\infty\right)  }\left\Vert f\left(  \tau_{y}\mathbf{1}_{\left[  0,1\right]
}\right)  ^{\frac{1}{p}}\right\Vert _{p}
%\left\Vert f\right\Vert _{p,\infty}
$. In fact, for this smaller norm, Proposition \ref{Equivalence1} fails for
all choices of the constant $C$. The reason for this is that the sup-norm of
$\tau_{y}\mathbf{1}_{\left[  0,1\right]  }$ tends to zero as $y\rightarrow
\infty$.

%equivalent, and in fact fails Proposition \ref{Equivalence1}
%for all choices of the constant $C$.

\end{remark}

\begin{proposition}
\label{Equivalence2}For $p\in\left[  1,\infty\right)  ,$%
\[
\left\Vert f\right\Vert _{p,\infty}\geq C\sup_{y\in\left[  0,\infty\right)
}\left(  \int\left\vert f\right\vert ^{p}\tau_{y}\mathbf{1}_{\left[
0,1\right]  }\,d\omega\right)  ^{1/p}.
\]

\end{proposition}

\begin{proof}
(i) For $y\in\left[  0,1\right)  $ the expression in (\ref{Three cases}) takes
the simpler form%
\[
\tau_{y}\mathbf{1}_{\left[  0,1\right]  }\left(  x\right)  =\left\{
\begin{array}
[c]{cc}%
1, & x\leq1-y,\\
& \\
\frac{1}{4xy}\left(  1-\left(  x-y\right)  ^{2}\right)  \leq1, & 1-y<x<1+y,\\
& \\
0, & x\geq1+y.
\end{array}
\right.
\]
Since $\tau_{y}\mathbf{1}_{\left[  0,1\right]  }\leq\mathbf{1}_{\left[
0,2\right)  }$ this gives%
\begin{align*}
\int\left\vert f\right\vert ^{p}\tau_{y}\mathbf{1}_{\left[  0,1\right]
}\,d\omega &  \leq\int\left\vert f\right\vert ^{p}\mathbf{1}_{\left[
0,1\right)  }\,d\omega+\int\left\vert f\right\vert ^{p}\mathbf{1}_{\left[
1,2\right)  }\,d\omega\\
&  \leq\frac{1}{\omega\left(  I_{1}\right)  }\int\left\vert f\right\vert
^{p}\mathbf{1}_{\left[  0,1\right)  }\,d\omega+\frac{3}{\omega\left(
I_{2}\right)  }\int\left\vert f\right\vert ^{p}\mathbf{1}_{\left[  1,2\right)
}\,d\omega\\
&  =\int_{I_{1}}\frac{1}{\omega\left(  I_{1}\right)  }\left\vert f\right\vert
^{p}\,d\omega+3\int_{I_{2}}\frac{1}{\omega\left(  I_{2}\right)  }\left\vert
f\right\vert ^{p}\,d\omega
\end{align*}
the second inequality holding since $\omega\left(  I_{1}\right)  =\frac{1}%
{3}<1$ and $\omega\left(  I_{2}\right)  =\frac{7}{3}<3.$ Hence%
\begin{align*}
\left(  \int\left\vert f\right\vert ^{p}\tau_{y}\mathbf{1}_{\left[
0,1\right]  }\,d\omega\right)  ^{1/p}  &  \leq\left(  \int_{I_{1}}\frac
{1}{\omega\left(  I_{1}\right)  }\left\vert f\right\vert ^{p}\,d\omega\right)
^{1/p}+3^{1/p}\left(  \int_{I_{2}}\frac{1}{\omega\left(  I_{2}\right)
}\left\vert f\right\vert ^{p}\,d\omega\right)  ^{1/p}\\
&  \leq\left(  1+3^{1/p}\right)  \left\Vert f\right\Vert _{p,\infty}%
\end{align*}

(ii) For $y\in\left[  1,2\right)  $ we have $\tau_{y}\mathbf{1}_{\left[
0,1\right]  }\leq\mathbf{1}_{\left[  0,3\right)  }$ which leads to%
\begin{align*}
&  \hspace{-2cm}\int\left\vert f\right\vert ^{p}\tau_{y}\mathbf{1}_{\left[
0,1\right]  }\,d\omega\\
&  \leq\int_{I_{1}}\left\vert f\right\vert ^{p}\,d\omega+\int_{I_{2}%
}\left\vert f\right\vert ^{p}\,d\omega+\int_{I_{3}}\left\vert f\right\vert
^{p}\,d\omega\\
&  \leq\int_{I_{1}}\frac{1}{\omega\left(  I_{1}\right)  }\left\vert
f\right\vert ^{p}\,d\omega+3\int_{I_{2}}\frac{1}{\omega\left(  I_{2}\right)
}\left\vert f\right\vert ^{p}\,d\omega+7\int_{I_{3}}\frac{1}{\omega\left(
I_{3}\right)  }\left\vert f\right\vert ^{p}\,d\omega
\end{align*}
since $\omega\left(  I_{3}\right)  <7,$\ and hence%
\begin{align*}
&  \hspace{-2cm}\left(  \int\left\vert f\right\vert ^{p}\tau_{y}%
\mathbf{1}_{\left[  0,1\right]  }\,d\omega\right)  ^{1/p}\\
&  \leq\left(  \int_{I_{1}}\frac{1}{\omega\left(  I_{1}\right)  }\left\vert
f\right\vert ^{p}\,d\omega\right)  ^{1/p}+3^{1/p}\left(  \int_{I_{2}}\frac
{1}{\omega\left(  I_{2}\right)  }\left\vert f\right\vert ^{p}\,d\omega\right)
^{1/p}+\\
&  +7^{1/p}\left(  \int_{I_{3}}\frac{1}{\omega\left(  I_{3}\right)
}\left\vert f\right\vert ^{p}\,d\omega\right)  ^{1/p}\\
&  \leq\left(  1+3^{1/p}+7^{1/p}\right)  \left\Vert f\right\Vert _{p,\infty}.
\end{align*}

(iii) For $y\geq2$ we have%
\[
\tau_{y}\mathbf{1}_{\left[  0,1\right]  }\left(  x\right)  =\left\{
\begin{array}
[c]{cc}%
\frac{1}{4xy}\left(  1-\left(  x-y\right)  ^{2}\right)  , & y-1<x<y+1,\\
& \\
0, & \text{otherwise.}%
\end{array}
\right.
\]
If $y\in I_{k},$ then $k\geq3$ and $\left(  y-1,y+1\right)  $ intersects at
most $I_{k-1},I_{k},I_{k+1}.$ For $x\in\left(  y-1,y+1\right)  $ we have%
\[
4xy>4\left(  y-1\right)  y>4\left(  k-2\right)  \left(  k-1\right)  .
\]
Now $k\geq3$ implies $4\left(  k-2\right)  \geq k$ and $3\left(  k-1\right)
\geq k+2$ so that%
\[
4xy\geq\frac{1}{3}k\left(  k+2\right)  \geq\frac{1}{3}\left(  k^{2}+k+\frac
{1}{3}\right)  =\frac{1}{3}\omega\left(  I_{k+1}\right)  \geq\frac{1}{3}%
\omega\left(  I_{k}\right)  \geq\frac{1}{3}\omega\left(  I_{k-1}\right)
\]
Thus we obtain for $j=k-1,k,k+1$%
\begin{align*}
\int_{I_{j}}\left\vert f\right\vert ^{p}\tau_{y}\mathbf{1}_{\left[
0,1\right]  }\,d\omega &  =\int_{I_{j}}\left\vert f\right\vert ^{p}\frac
{1}{4xy}\left(  1-\left(  x-y\right)  ^{2}\right)  \mathbf{1}_{\left(
y-1,y+1\right)  }\,d\omega\\
&  \leq3\int_{I_{j}}\frac{1}{\omega\left(  I_{j}\right)  }\left\vert
f\right\vert ^{p}\,d\omega
\end{align*}
and%
\[
\left(  \int\left\vert f\right\vert ^{p}\tau_{y}\mathbf{1}_{\left[
0,1\right]  }\,d\omega\right)  ^{1/p}\leq3^{1/p}\sum_{j=k-1}^{k+1}\left(
\int_{I_{j}}\frac{1}{\omega\left(  I_{j}\right)  }\left\vert f\right\vert
^{p}\,d\omega\right)  ^{1/p}\leq3^{1+1/p}\left\Vert f\right\Vert _{p,\infty}.
\]

(iv) Taking $C$ to be the maximum of the constants in (i)-(iii) we have%
\[
\left(  \int\left\vert f\right\vert ^{p}\tau_{y}\mathbf{1}_{\left[
0,1\right]  }\,d\omega\right)  ^{1/p}\leq C\left\Vert f\right\Vert _{p,\infty}%
\]
for all $y\in\left[  0,\infty\right)  $ and hence%
\[
\sup_{y\in\left[  0,\infty\right)  }\left(  \int\left\vert f\right\vert
^{p}\tau_{y}\mathbf{1}_{\left[  0,1\right]  }\,d\omega\right)  ^{1/p}\leq
C\left\Vert f\right\Vert _{p,\infty}.
\]

\hspace{12cm}
\end{proof}

\subsection{Functions that are square integrable on a neighbourhood of the
identity\label{SquareInt}}

\label{fournier}

For $p=2$ we have the following characterisation along the lines of
\cite{Fournier1}, Theorem 3.1.

\begin{theorem}
\label{transforms} For $f\in L^{1}\left(  \mathbb{R}_{+},\ast_{\frac{1}{2}%
}\right)  $ with $\overset{\wedge}{f}\geq0$ the following are equivalent:

\begin{enumerate}
\item $f$ is square integrable in a neighbourhood of the identity;

\item $\overset{\wedge}{f}\in\left(  L^{1},\ell^{2}\right)  \left(
\mathbb{R}_{+},\ast_{\frac{1}{2}}\right)  ;$

\item $f\in\left(  L^{2},\ell^{\infty}\right)  \left(  \mathbb{R}_{+}%
,\ast_{\frac{1}{2}}\right)  .$
\end{enumerate}
\end{theorem}

\begin{proof}
The proof of Theorem 3.1 in \cite{Fournier1} applies, but we need to check
that the results from \cite{Fournier 2}, \cite{FournierStewart}\ and
\cite{Holland}\ are still valid in our setting.
%We leave the detailed checking to the reader.
This requires the equivalence of the continuous\ and the discrete\ amalgam
norms, which we showed in Propositions \ref{Equivalence1} and
\ref{Equivalence2}, together with uniform boundedness of translation along
with the Hausdorff-Young theorem for these amalgam spaces. We prove these
properties in the next three sections.
\end{proof}

\subsection{Translation in $(L^{\infty},\ell^{1})(\mathbb{R}_{+},\ast
_{\frac{1}{2}})\label{Translation}$}

In this section we show that translation is uniformly bounded on the amalgam
space $(L^{\infty},\ell^{1})(\mathbb{R}_{+},\ast_{\frac{1}{2}})$. Denote the
Haar measure $\omega\left(  I_{n}\right)  $ of the interval $I_{n}\ $by
$\omega_{n}$. It is easily checked that $\omega_{n}=n^{2}-n+\frac{1}{3}.$
Given a locally integrable function $f$ on $\mathbb{R}_{+}$ let $P_{n}%
f:=f\mathbf{1}_{I_{n}}$ and consider%
\[
\tau_{y}f\left(  x\right)  =f\left(  x\ast_{\frac{1}{2}}y\right)  =\frac
{1}{2xy}\int_{\left[  \left\vert x-y\right\vert ,x+y\right]  }f(t)t\,dt.
\]
Note that $|\tau_{y}f|\leq\tau_{y}(|f|)$ pointwise, and that $\tau
_{y}(|f|)\leq\tau_{y}g$ if $|f|\leq g$ almost everywhere. We want to show
uniform boundedness of the translation operators $\tau_{y}$\ on $(L^{\infty
},\ell^{1})(\mathbb{R}_{+},\ast_{\frac{1}{2}}).$

Consider an index $n$ and a positive number $y$, and write $f_{n}%
:=\mathbf{1}_{I_{n}}$ . It will be enough to show that
\[
\Vert\tau_{y}f_{n}\Vert_{(L^{\infty},\ell^{1})}\leq C\Vert f_{n}%
\Vert_{(L^{\infty},\ell^{1})}%
\]
for a number $C$ that is independent of $y$ and $n$. Indeed, letting
$c_{n}=\Vert P_{n}f\Vert_{\infty}$ and $g=\sum_{n}c_{n}f_{n}$, we then have
that $|f|\leq g$ pointwise, and thus $\Vert\tau_{y}f\Vert_{(L^{\infty}%
,\ell^{1})}\leq\Vert\tau_{y}g\Vert_{(L^{\infty},\ell^{1})}$. But also
$\tau_{y}g\leq\sum_{n}c_{n}\tau_{y}(f_{n})$ pointwise so that
\[
\Vert\tau_{y}f\Vert_{(L^{\infty},\ell^{1})}\leq\Vert\tau_{y}g\Vert
_{(L^{\infty},\ell^{1})}\leq\sum_{n}c_{n}\Vert\tau_{y}f_{n}\Vert_{(L^{\infty
},\ell^{1})}%
\]%
\[
\leq\sum_{n}c_{n}C\Vert f_{n}\Vert_{(L^{\infty},\ell^{1})}=C\Vert
f\Vert_{(L^{\infty},\ell^{1})}.
\]

Fix $y$ and $n$, and call a non-negative integer $k$ \emph{exceptional} if
$k=1$ or if there is some number $x$ in the interval $I_{k}$ such that $|x-y|$
or $x+y$ lies in $I_{n}$. Denote the set of exceptional indices by $E$, and
let $G$ be the set of \emph{generic} indices forming the complement of $E$ in
$\mathbb{Z}_{+}$.

If $k$ is generic, then the intersection of the interval $[|x-y|,x+y]$ with
$I_{n}$ is either empty for all $x$ in $I_{k}$, or this intersection is all of
$I_{n}$ for all such $x$. Then $\tau_{y}f_{n}$ either vanishes on the whole
interval $I_{k}$ or it coincides on $I_{k}$ with
\begin{equation}
\frac{1}{2xy}\int_{n-1}^{n}t\,dt. \label{LocalTranslation}%
\end{equation}
Since $k\geq2$, the expression above does not change by more than a factor of
$2$ as $x$ runs through the interval $I_{k}$.

So for each generic index $k$ there is a non-negative constant $d_{k}$ with
$d_{k}\leq\tau_{y}f_{n}(x)\leq2d_{k}$ for all $x$ in $I_{k}$. Then
\[
\omega_{k}\Vert P_{k}(\tau_{y}f_{n})\Vert_{\infty}\leq\omega_{k}2d_{k}%
\leq2\Vert P_{k}(\tau_{y}f_{n})\Vert_{1}.
\]
Note too that $\omega_{n}\Vert f_{n}\Vert_{\infty}=\Vert f_{n}\Vert_{1}$ since
$f_{n}$ is constant $\left(  =1\right)  $\ on its support $I_{n}$. Therefore,
\[
\sum_{k\in G}\omega_{k}\Vert P_{k}(\tau_{y}f_{n})\Vert_{\infty}\leq\sum_{k\in
G}2\Vert P_{k}(\tau_{y}f_{n})\Vert_{1}\leq\sum_{k\in\mathbb{Z}_{+}}2\Vert
P_{k}(\tau_{y}f_{n})\Vert_{1}%
\]%
\[
=2\Vert\tau_{y}(f_{n})\Vert_{1}\leq2\Vert f_{n}\Vert_{1}=2\omega_{n}\Vert
f_{n}\Vert_{\infty}=2\Vert f_{n}\Vert_{(L^{\infty},\ell^{1})},
\]
the last inequality holding since translation is bounded, with norm $1$, on
$L^{1}(\mathbb{R}_{+},\ast_{\frac{1}{2}})$.

One way for $k$ to be exceptional is to have $x+y$ belong to $I_{n}$ for some
$x$ in $I_{k}$, that is, the sets $y+I_{k}$ and $I_{n}$ have non-empty
intersection; equivalently, the set $I_{n}-y$ overlaps $I_{k}$. There are at
most two such values of $k$, and none when $y>n$. Any other exceptional
indices $k$ come from cases where $I_{n}+y$ or $y-I_{n}$ overlap $I_{k}$, or
$k=1$. It follows easily that there are at most seven exceptional indices, and
in fact there are at most five of them.

It remains to estimate $\omega_{k}\Vert P_{k}(\tau_{y}f_{n})\Vert_{\infty}$
for each exceptional index $k$. When $k\leq3n$ use the estimate
\[
\tau_{y}f_{n}\left(  x\right)  \leq\frac{1}{2xy}\int_{\left[  \left\vert
x-y\right\vert ,x+y\right]  }t\,dt=\frac{1}{4xy}\{(x+y)^{2}-|x-y|^{2}\}=1
\]
to see that
\[
\omega_{k}\Vert P_{k}(\tau_{y}f_{n})\Vert_{\infty}\leq\omega_{k}\leq
\omega_{3n}\leq19\omega_{n}=19\Vert f_{n}\Vert_{(L^{\infty},\ell^{1})}.
\]

When $k$ is exceptional and $k>3n$, one of the sets $y\pm I_{n}$ must overlap
$I_{k}$. The smallest value that $y$ could take would then satisfy $y+n=k-1$,
making $y+\frac{1}{3}k>k-1$ and $y>\frac{2}{3}k-1>\frac{1}{3}k$ since $k>3$.
In particular, $y>\frac{1}{3}x$ for all $x$ in $I_{k}$ in these cases. For
this $k$ and such $x$ use the upper bound
\[
\tau_{y}f_{n}\left(  x\right)  \leq\frac{1}{2xy}\int_{n-1}^{n}t\,dt=\frac
{1}{4xy}\{n^{2}-(n-1)^{2}\}<\frac{2n}{x^{2}}\leq\frac{2n}{\left(  k-1\right)
^{2}}%
\]
where the first inequality follows from (\ref{LocalTranslation}),\ to see
that
\[
\omega_{k}\Vert P_{k}(\tau_{y}f_{n})\Vert_{\infty}\leq\frac{k^{2}%
(2n)}{(k-1)^{2}}\leq8n\leq24\omega_{n}\leq24\Vert f_{n}\Vert_{(L^{\infty}%
,\ell^{1})}.
\]

\subsection{Translation and convolution on $(L^{p},\ell^{q})(\mathbb{R}%
_{+},\ast_{\frac{1}{2}})\label{Convolution}$}

In this section we deduce that translation is uniformly bounded on
$(L^{p},\ell^{q})(\mathbb{R}_{+},\ast_{\frac{1}{2}})$ and note that Young's
inequality for convolution also holds for the amalgam spaces on $(\mathbb{R}%
_{+},\ast_{\frac{1}{2}})$. The uniform boundedness of translation on
$(L^{\infty},\ell^{1})(\mathbb{R}_{+},\ast_{\frac{1}{2}})$ implies by duality
that it also holds on $(L^{1},\ell^{\infty})(\mathbb{R}_{+},\ast_{\frac{1}{2}%
})$. To confirm this, first note that matters reduce to the case of a
non-negative function, $g$ say, in $(L^{1},\ell^{\infty})(\mathbb{R}_{+}%
,\ast_{\frac{1}{2}})$, and that $\tau_{y}g$ is then also non-negative. This
translate belongs to $(L^{1},\ell^{\infty})(\mathbb{R}_{+},\ast_{\frac{1}{2}%
})$ if and only if
\[
\int_{\mathbb{R}_{+}}\left(  \tau_{y}g(x)\right)  f(x)\,d\omega(x)<\infty
\]
for all non-negative functions $f$ in the unit ball of $(L^{\infty},\ell
^{1})(\mathbb{R}_{+},\ast_{\frac{1}{2}})$. In this case, the norm of $\tau
_{y}g$ in $(L^{1},\ell^{\infty})(\mathbb{R}_{+},\ast_{\frac{1}{2}})$ is equal
to the supremum of these integrals over all such functions $f$. By
\cite{Bloom/Heyer}, Theorem 1.3.21, and the fact that $y^{-}=y$, these
integrals are equal to%
\[
\int_{\mathbb{R}_{+}}g(z)\left(  \tau_{y}f(z)\right)  \,d\omega(z)\leq\Vert
g\Vert_{(L^{1},\ell^{\infty})}\Vert\tau_{y}f\Vert|_{(L^{\infty},\ell^{1})}\leq
C\Vert g\Vert_{(L^{1},\ell^{\infty})}.
\]
We thus have uniform boundedness of translation on $(L^{p},\ell^{q}%
)(\mathbb{R}_{+},\ast_{\frac{1}{2}})$ when the reciprocal indices $(1/p,1/q)$
sit at any of the four corners of the unit square in the first quadrant. As in
\cite{Holland}, complex interpolation then yields uniform boundedness of
translation whenever $(1/p,1/q)$ lies in this unit square, that is whenever
$1\leq p,q\leq\infty$. This also follows in a more elementary way from
H\"{o}lder's inequality.

As in the case of locally compact abelian groups, Young's inequality for
convolution of $L^{p}$-functions extends to these amalgams. The general
statement is that if reciprocal indices in the unit square satisfy the
condition%
\[
\left(  \frac{1}{p},\frac{1}{q}\right)  =\left(  \frac{1}{p_{1}},\frac
{1}{q_{1}}\right)  +\left(  \frac{1}{p_{2}},\frac{1}{q_{2}}\right)  -(1,1)
\]
and if functions $f_{1}$ and $f_{2}$ belong to the respective amalgams
$(L^{p_{1}},\ell^{q_{1}})(\mathbb{R}_{+},\ast_{\frac{1}{2}})$ and $(L^{p_{2}%
},\ell^{q_{2}})(\mathbb{R}_{+},\ast_{\frac{1}{2}})$, then the convolution of
$f_{1}$ and $f_{2}$ is defined and belongs to $(L^{p},\ell^{q})(\mathbb{R}%
_{+},\ast_{\frac{1}{2}})$. Moreover, we have%
\[
\Vert f_{1}\ast_{\frac{1}{2}}f_{2}\Vert_{(L^{p},\ell^{q})}\leq C\Vert
f_{1}\Vert_{(L^{p_{1}},\ell^{q_{1}})}\Vert f_{2}\Vert_{(L^{p_{2}},\ell^{q_{2}%
})}.
\]
In fact, the inclusions between amalgams then imply that these statements
still hold, usually with a different constant $C$, provided that
$1/p\leq1/p_{1}+1/p_{2}-1$ and $1/q\geq1/q_{1}+1/q_{2}-1$. Another way to
state this is that $(1/p,1/q)$ can be any point in the unit square lying
northwest of the point $(1/p_{1}+1/p_{2}-1,1/q_{1}+1/q_{2}-1)$, which is also
required to lie in the unit square. Again the general case follows from a few
extreme cases by complex interpolation or by repeated use of H\"{o}lder's inequality.

\subsection{ \textbf{Fourier transforms on }$(L^{p},\ell^{q})(\mathbb{R}%
_{+},\ast_{\frac{1}{2}})\label{FourierTransform}$}

Our goal in this section is to prove that if $f\in(L^{p},\ell^{q}%
)(\mathbb{R}_{+},\ast_{\frac{1}{2}})$ with $1\leq p,q\leq2$, then
$\overset{\wedge}{f}\in(L^{p^{\prime}},\ell^{q^{\prime}})(\mathbb{R}_{+}%
,\ast_{\frac{1}{2}})$. The cases where $p=q$ are already known (see
\cite{Degenfeld-Schonberg}) with the same proof as for locally compact abelian
groups, but if $p\neq q$, then this property of the Fourier transform requires
some work. These cases will follow by complex interpolation from those where
$p=q$ and the special ones where $(p,q)=(2,1)$ or $(1,2)$. (The latter is the
one that arises in the proof of Theorem \ref{transforms}.) We show below that
the two special cases are equivalent by duality, and we prove the first case
using some easily-checked properties of transforms of the indicator functions
$\mathbf{1}_{I_{n}}$.

From (\ref{IndicatorFourierTransform}) we find that the Fourier transform of
$\mathbf{1}_{I_{1}}$ belongs to $(L^{\infty},\ell^{q})$ $(\mathbb{R}_{+}%
,\ast_{\frac{1}{2}})$ for all $q>\frac{3}{2}$, but does \emph{not} belong to
$(L^{p},\ell^{1})(\mathbb{R}_{+},\ast_{\frac{1}{2}})$ for any value of $p$.
Let $g_{1}=3\,\mathbf{1}_{I_{1}}\ast_{\frac{1}{2}}\mathbf{1}_{[0,2)}$ and
$g_{n}=3\,\mathbf{1}_{I_{1}}\ast_{\frac{1}{2}}\mathbf{1}_{[n-2,n+1)}$ when
$n>1$. We can check that $g_{n}(x)=1$ for all $x$ in $I_{n}$. When $n>1$,
H\"{o}lder's inequality gives
\begin{align*}
\Vert\overset{\wedge}{g_{n}}\Vert_{(L^{2},\ell^{1})}  &  =3\left\Vert
\widehat{\mathbf{1}_{I_{1}}}\left(  \widehat{\mathbf{1}_{I_{n-1}}%
}+\widehat{\mathbf{1}_{I_{n}}}+\widehat{\mathbf{1}_{I_{n+1}}}\right)
\right\Vert _{(L^{2},\ell^{1})}\\
&  \leq3\left\Vert \widehat{\mathbf{1}_{I_{1}}}\right\Vert _{(L^{\infty}%
,\ell^{2})}\left\Vert \widehat{\mathbf{1}_{I_{n-1}}}+\widehat{\mathbf{1}%
_{I_{n}}}+\widehat{\mathbf{1}_{I_{n+1}}}\right\Vert _{(L^{2},\ell^{2})}\\
&  =C\left\Vert \widehat{\mathbf{1}_{I_{n-1}}}+\widehat{\mathbf{1}_{I_{n}}%
}+\widehat{\mathbf{1}_{I_{n+1}}}\right\Vert _{2}\\
&  =C\left\Vert \mathbf{1}_{I_{n-1}}+\mathbf{1}_{I_{n}}+\mathbf{1}_{I_{n+1}%
}\right\Vert _{2}\\
&  =C(\omega_{n-1}+\omega_{n}+\omega_{n+1})^{1/2}\\
&  \leq C^{\prime}\sqrt{\omega_{n}}.
\end{align*}

%Recall that in computing norms in $(L^{p},\ell^{q})(\mathbb{R}_{+},\ast_{\frac{1}{2}}),$ when $p,q$ are finite we divide by $\omega(I_{n})$ in the integral leading to the local $L^{p}$-norm, then we take the $p$-th root and $q$-th power, multiply by $\omega_{n}$, sum over $n,$ and finally take the $q$-th root.
%In particular,
By formula \eqref{DiscreteNorm}, if $f\in L^{2}(\mathbb{R}_{+},\ast_{\frac
{1}{2}})$ and
%if
$f$ vanishes outside $I_{n}$, then $\Vert f\Vert_{(L^{2},\ell^{1})}%
=\sqrt{\omega_{n}}\Vert f\Vert_{2}.$ Moreover, in this case $f=fg_{n}$ and it
follows by Young's inequality for convolution of amalgams that
\[
\left\Vert \overset{\wedge}{f}\right\Vert _{(L^{\infty},\ell^{2})}=\left\Vert
\widehat{fg_{n}}\right\Vert _{(L^{\infty},\ell^{2})}=\left\Vert
\overset{\wedge}{f}\ast_{\frac{1}{2}}\overset{\wedge}{g_{n}}\right\Vert
_{(L^{\infty},\ell^{2})}\leq\left\Vert \overset{\wedge}{f}\right\Vert
_{(L^{2},\ell^{2})}\left\Vert \overset{\wedge}{g_{n}}\right\Vert _{(L^{2}%
,\ell^{1})}%
\]%
\[
\leq\left\Vert \overset{\wedge}{f}\right\Vert _{2}C\sqrt{\omega_{n}}%
=C\sqrt{\omega_{n}}\Vert f\Vert_{2}=C\Vert f\Vert_{(L^{2},\ell^{1})}.
\]
For a general function $f$ in $(L^{2},\ell^{1})(\mathbb{R}_{+},)$, applying
the inequalities above to
%again let
$P_{n}f
%$,
%again defined to be $
:=f\mathbf{1}_{I_{n}}$ yields that $\left\Vert \widehat{P_{n}f}\right\Vert
_{(L^{\infty},\ell^{2})}\leq C\sqrt{\omega_{n}}\left\Vert P_{n}f\right\Vert
_{2}$. Since for $(p,q)=(2,1)$, formula \eqref{DiscreteNorm} takes the special
form
%simplifies to give that
%\[
$\left\Vert f\right\Vert _{(L^{2},\ell^{1})}=\sum_{n=1}^{\infty}\sqrt
{\omega_{\alpha}\left(  I_{n}\right)  }\left\Vert P_{n}f\right\Vert _{2},$
%\]
it follows that $\left\Vert \hat{f}\right\Vert _{(L^{\infty},\ell^{2})}\leq
C\left\Vert f\right\Vert _{(L^{2},\ell^{1})}$.

Suppose next that $g\in L^{1}(\mathbb{R}_{+},\ast_{\frac{1}{2}})$. Then the
function $\overset{\wedge}{g}$ belongs to $(L^{2},\ell^{\infty})(\mathbb{R}%
_{+},\ast_{\frac{1}{2}})$ if and only if $\overset{\wedge}{g}f\in
L^{1}(\mathbb{R}_{+},\ast_{\frac{1}{2}})$ for all functions $f$ in the unit
ball of $(L^{2},\ell^{1})(\mathbb{R}_{+},\ast_{\frac{1}{2}})$. In this case,
$\Vert\overset{\wedge}{g}\Vert_{(L^{2},\ell^{1})}$ is equal to the supremum
over all such functions $f$ of the numbers $|\int\overset{\wedge
}{g}(t)f(t)\omega(t)\,dt|$. But each of these integrals is equal to $\int
g(x)\overset{\wedge}{f}(x)\omega(x)\,dx$ and so has absolute value less than
or equal to%
\[
\Vert g\Vert_{(L^{1},\ell^{2})}\left\Vert \overset{\wedge}{f}\right\Vert
_{(L^{\infty},\ell^{2})}\leq\Vert g\Vert_{(L^{1},\ell^{2})}C\Vert
f\Vert_{(L^{2},\ell^{1})}=C\Vert g\Vert_{(L^{1},\ell^{2})}.
\]
In other words, the Fourier transform is a bounded operator from
$L^{1}(\mathbb{R}_{+},\ast_{\frac{1}{2}})$ to $(L^{2},\ell^{\infty
})(\mathbb{R}_{+},\ast_{\frac{1}{2}})$ when $L^{1}(\mathbb{R}_{+},\ast
_{\frac{1}{2}})$ is viewed as a dense subspace of $(L^{1},\ell^{2}%
)(\mathbb{R}_{+},\ast_{\frac{1}{2}})$ with the norm $\Vert\cdot\Vert
_{(L^{1},\ell^{2})}$. Extend this operator to all of $(L^{1},\ell
^{2})(\mathbb{R}_{+},\ast_{\frac{1}{2}})$.

This includes the usual extension of the Fourier transform operator from
$L^{1}(\mathbb{R}_{+},\ast_{\frac{1}{2}})\cap L^{2}(\mathbb{R}_{+},\ast
_{\frac{1}{2}})$ to an isometry from the space $L^{2}(\mathbb{R}_{+}%
,\ast_{\frac{1}{2}})$ to a dual copy of $L^{2}(\mathbb{R}_{+},\ast_{\frac
{1}{2}})$. It also includes the transform originally defined as a mapping of
$L^{1}(\mathbb{R}_{+},\ast_{\frac{1}{2}})$ to $L^{\infty}(\mathbb{R}_{+}%
,\ast_{\frac{1}{2}})$ and shown above to map the smaller space $(L^{2}%
,\ell^{1})(\mathbb{R}_{+},\ast_{\frac{1}{2}})$ to $(L^{\infty},\ell
^{2})(\mathbb{R}_{+},\ast_{\frac{1}{2}})$. So, the Hausdorff-Young theorem
holds for amalgams in the four extreme cases where the indices $(p,q)$ are
$(1,1)$, $(2,2)$, $(2,1)$ and $(1,2)$, and the other cases then follow by
complex interpolation.

\section{Some countable non-discrete hypergroups\label{Countable}}

%The conclusion of Wiener's theorem also holds for non-even exponents $p$
%on some commutative hypergroups.
The positive conclusion in Wiener's theorem also holds for non-even exponents
in the interval $\left[  1,\infty\right)  $ on some countable compact
hypergroups $H_{a}$ considered in \cite{DunklRamirez} and \cite{Spector}, and
on the countable locally compact hypergroup $H$ below. Here $a$ is a parameter
in the interval $(0,1/2]$. We let $a=1/2$ and leave the other cases for the reader.

\subsection{Compact countable commutative hypergroups}

\begin{example}
\label{DR} The one-point compactification $\mathbb{Z}_{+}\cup\left\{
\infty\right\}  $ of the non-negative integers is a compact commutative
hypergroup $\left(  H_{\frac{1}{2}},\ast\right)  $ with convolution given by%
\begin{equation}
\varepsilon_{m}\ast\varepsilon_{n}=\left\{
\begin{array}
[c]{cc}%
\sum_{k=1}^{\infty}\frac{1}{2^{k}}\varepsilon_{k+n}, & m=n\in\mathbb{Z}_{+},\\
& \\
\varepsilon_{\infty}, & m=n=\infty,\\
& \\
\varepsilon_{\min\left\{  m,n\right\}  }, & m\neq n\in\mathbb{Z}%
_{+}\mathbb{\cup\{ \infty\}},
\end{array}
\right.
\end{equation}
so that $\varepsilon_{\infty}$\ is the identity element. The Haar measure
$\omega$\ is given by $\omega\left(  n\right)  =\frac{1}{2^{n+1}}$ for
$n<\infty$ and $\omega\left(  \infty\right)  =0.$ The characters $\chi_{n}$
are given by%
\[
\chi_{n}\left(  m\right)  =\left\{
\begin{array}
[c]{cc}%
0, & m\leq n-2,\\
-1, & m=n-1,\\
1, & m\geq n
\end{array}
\right.
\]
where $n\in\mathbb{Z}_{+}$, and the Plancherel measure $\pi$ is just%
\begin{equation}
\pi(\chi_{n})=\frac{1}{\Vert\chi_{n}\Vert_{2}^{2}}=%
\begin{cases}
2^{n-1}, & \text{if }n\text{$\geq1$,}\\
1, & \text{if }n\text{$=0$.}%
\end{cases}
\label{Plancherel}%
\end{equation}

We observe that the set\ of continuous positive definite functions is given by%
\begin{equation}
P\left(  H_{\frac{1}{2}}\right)  =\left\{  f:f=\sum_{i=0}^{\infty}\alpha
_{i}\chi_{i}:\alpha_{i}\geq0,\sum_{i=0}^{\infty}\alpha_{i}<\infty\right\}
\label{Series expansion}%
\end{equation}
(indeed, in $\cite{DunklRamirez}$, equation (\ref{Series expansion}) is the
definition of $P\left(  H_{1/2}\right)  $). It is a consequence of Bochner's
theorem ($\cite{Bloom/Heyer}$, Theorems 4.1.15 and 4.1.16) that
(\ref{Series expansion}) holds if and only if $f\in P_{b}\left(  H_{\frac
{1}{2}}\right)  $, and this space coincides with $P\left(  H_{\frac{1}{2}%
}\right)  $ because $H_{\frac{1}{2}}$\ is compact.

If $f$ is as in (\ref{Series expansion})\ then%
\begin{equation}
f\left(  n\right)  =\left(  \sum_{i=0}^{n}\alpha_{i}\right)  -\alpha_{n+1}
\label{PositiveDefinite}%
\end{equation}
for $n\in\mathbb{Z}_{+}$ and (because of continuity)%
\begin{equation}
f\left(  \infty\right)  =\sum_{i=0}^{\infty}\alpha_{i}. \label{Infinity}%
\end{equation}

\begin{remark}
For $f\in P\left(  H_{\frac{1}{2}}\right)  $ we have $\left\Vert f\right\Vert
_{\infty}=f\left(  \infty\right)  $, as seen from (\ref{PositiveDefinite}) and
(\ref{Infinity}) (or from $\cite{Bloom/Heyer}$, Lemma 4.1.3(g)).
\end{remark}
\end{example}

\subsection{ Operations on $P\left(  H_{\frac{1}{2}}\right)  $}

By (\ref{Series expansion}) the function $f$ is the inverse Fourier transform
of%
\[
i\longmapsto\alpha_{i}/\pi\left(  \chi_{i}\right)
\]
and the latter function (on $
%\left.
\widehat{H_{\frac{1}{2}}}
%\! \right.  ^{\wedge}
$) belongs to $L^{1}\left(  \pi\right)  $. The set of inverse transforms of
functions in $L^{1}\left(  \pi\right)  $ is called the Fourier algebra of
$H_{\frac{1}{2}}$, and is denoted by $A\left(  \left.  H_{\frac{1}{2}}\!
\right.  \right)  $.
%In formula (\ref{Series expansion}), $\alpha_i = \pi(i)\hat f(i)$.
It is shown in \cite{DunklRamirez}
%then imply
that Lipschitz functions operate
%the absolute value operates
on $A\left(  \left.  H_{\frac{1}{2}}\! \right.  \right)  $;
%, that is
in particular, if $f\in A\left(  \left.  H_{\frac{1}{2}}\! \right.  \right)  $
and $1\leq p<\infty$, then $|f|^{p}\in A\left(  \left.  H_{\frac{1}{2}}\!
\right.  \right)  $
%$|f|\in A\left(  \left.  H_{\frac{1}{2}}\! \right.\right)  $
as well. We
%consider
%need
%now
prove the corresponding statement for $P\left(  H_{\frac{1}{2}}\right)  $ and
apply it in Section \ref{WienerCountableCompact}.
%the corresponding coefficients of $\left \vert f\right \vert $ belong to $\ell^{1}$; we will show in the next section that they are non-negative.

\begin{proposition}
\label{operate} Let $1 \le p < \infty$. Suppose that $f: H_{\frac{1}{2}}
\rightarrow\mathbb{C}$ is $p$-integrable in a neighbourhood $U$ of the
identity $e$. If $f$ is of positive type
%and $1 \le p < \infty$,
then so is $\left\vert f\right\vert ^{p} $. In particular, if $f\in P\left(
H_{\frac{1}{2}}\right)  \ $then $\left\vert f\right\vert ^{p} \in P\left(
H_{\frac{1}{2}}\right)  $.
%that is,
%$|\cdot|^p$
%the absolute value
%operates on $P\left(  H_{\frac{1}{2}}\right)  .$

\end{proposition}

\begin{proof}
The
%the
%local
$p$-integrability of $f$ near $e$ implies global $p$-integrability, because
the complement of $U$
%$H_{\frac{1}{2}} \backslash U$
is finite. Since the Plancherel measure has full support, Remark
\ref{Positive type} then reduces matters to checking that the Fourier
coefficients of $\left\vert f\right\vert ^{p} $\
%$\left \vert f\right \vert $\
are non-negative if those of $f$ are.

When $p=1$, let $r(n)=f\left(  n\right)  \omega\left(  n\right)  $ for each
$n$; then $r\in\ell^{1}$ since $f$ is integrable. We claim that
$\overset{\wedge}{f}\geq0$ if and only if $r$ is real-valued and%
\begin{equation}
|r(n)|\leq r(n+1)+r(n+2)+\cdots\quad\text{for all $n$} \label{Tails}%
\end{equation}
If these inequalities hold for $f$, then they also hold when all negative
values $r(m)$ are replaced by $|r(m)|$, that is when $f$ is replaced by $|f|$.
So the case of the proposition where $p=1$ follows from our claim.

The conditions above on $r$ are equivalent to requiring for all $n$ that
\begin{equation}
r(n)+r(n+1)+r(n+2)+\cdots\geq0 \label{plus}%
\end{equation}
and
\begin{equation}
-r(n)+r(n+1)+r(n+2)+\cdots\geq0. \label{minus}%
\end{equation}
Indeed, subtracting the two inequalities for the same value of $n$ shows that
$r(n)$ is real, and then inequality (\ref{Tails}) follows since $|r(n)|=\max
\{r(n),-r(n)\}$. The converse is obvious.

Condition (\ref{minus}) is equivalent to requiring that $\overset{\wedge
}{f}(n+1)\geq0$, while the $0^{th}$ case of condition (\ref{plus}) is
equivalent to requiring that $\overset{\wedge}{f}(0)\geq0$. If condition
(\ref{minus}) holds for all $n$, and condition (\ref{plus}) holds for some
value of $n$, then adding the corresponding case of condition (\ref{minus})
shows that condition (\ref{plus}) also holds for the next value of $n$. So the
two conditions hold of all values of $n$ if and only if $f$ is of positive type.

To deal with exponents $p$ in the interval $(1, \infty)$,
%first note that by assumption $f$ is $p$-integrable, since $H_{\frac{1}{2}} \backslash U$ is finite.
consider
%Consider
the $n$-th instance of condition \eqref{Tails} with $f$ replaced by $|f|^{p}$,
that is
\[
|f(n)|^{p}\omega(n) \leq|f(n+1)|^{p}\omega(n+1) + |f(n+2)|^{p}\omega(n+2)
+\cdots.
%\quad \text{for all $n$.}
\]
Let~$\omega_{n}(n+k) = \omega(n+k)/\omega(n)$ when $k = 1, 2 \cdots$. The
inequality above is equivalent to requiring that
\begin{equation}
\label{ReTails}|f(n)| \le\left[  \sum_{k=1}^{\infty}|f(n+k)|^{p}\omega
_{n}(n+k)\right]  ^{1/p}.
\end{equation}
%\]
The expression on the right above is the $L^{p}$ norm of the restriction of
$f$ to the set $\{ n+1, n+2, \cdots\}$ with respect to the measure $\omega
_{n}$, which has total mass $1$. By H\"older's inequality, that $L^{p}$ norm
majorizes the corresponding $L^{1}$ norm. So it is enough the prove inequality
\eqref{ReTails}
%the last inequality displayed above
when $p=1$, and that was done in the first part of the proof.
\end{proof}

\subsection{A locally compact example}

We now analyse a non-compact example presented in \cite{Spector}.
%
%\begin{example}
\label{Non-compact}
%Rename $H_{1/2}$ as $H_{1/2,0}$ to emphasise the fact that
%it coincides with the set $U_{0}$.
For $N>0$ the set $U_{N}$ defined by
\begin{equation}
U_{N}:=\{N,N+1,N+2,\cdots,\infty\} \label{Compact neighbourhoods}%
\end{equation}
is a proper subhypergroup of $H_{1/2}$ and is isomorphic to $H_{1/2}$, but
with a scaled Haar measure. Define similar hypergroups $U_{N}$ when $N\leq0$
($U_{0}=H_{1/2}$), and let $H$ be the union of these nested compact
hypergroups. Then $H$ is a locally compact commutative hypergroup with
convolution given by%
\begin{equation}
\varepsilon_{m}\ast\varepsilon_{n}=\left\{
\begin{array}
[c]{cc}%
\sum_{k=1}^{\infty}\frac{1}{2^{k}}\varepsilon_{k+n}, & m=n\in\mathbb{Z},\\
& \\
\varepsilon_{\infty}, & m=n=\infty,\\
& \\
\varepsilon_{\min\left\{  m,n\right\}  }, & m\neq n\in\mathbb{Z\cup\{
\infty\}},
\end{array}
\right.  \label{NoncompactConvolution}%
\end{equation}
so that $\varepsilon_{\infty}$\ is the identity element, but $H$ is not compact.

The functions $\chi_{n}$ in Example \ref{DR}, with $n$ now allowed to be any
integer, comprise all the characters on $H$ except for the character
$\chi_{-\infty}\equiv1$, which has Plancherel measure $0$. The first case of
formula (\ref{Plancherel}) for the Plancherel measure of $\chi_{n}$ extends to
all indices $n\leq0$ (in particular we now have $\pi\left(  \chi_{0}\right)
=\frac{1}{2}$).

Note that $H$ is Pontryagin since (up to the different parametrization of
$H^{\wedge}$) it is self-dual via the mapping $n\rightarrow\chi_{-n}$. In fact
it is straightforward to see that%
\[
\chi_{m}\chi_{n}=\left\{
\begin{tabular}
[c]{ll}%
$\sum_{k=1}^{\infty}\tfrac{1}{2^{k}}\chi_{n-k},$ & $m=n\in\mathbb{Z},$\\
& \\
$\chi_{-\infty},$ & $m=n=-\infty,$\\
& \\
$\chi_{\max\left\{  m,n\right\}  },$ & $m\neq n\in\mathbb{Z\cup\{-\infty\}.}$%
\end{tabular}
\right.
\]

%\end{example}

\begin{remark}
\label{PositiveDefiniteproductextended}By $\cite{Bloom/Heyer}$, Corollary
2.4.20(ii), $H_{\frac{1}{2}}$ is also Pontryagin. In particular, $H$ and
$H_{\frac{1}{2}}$ are strong hypergroups (that is, their canonical duals are
also hypergroups). Now use Remark \ref{PositiveDefiniteproduct} to obtain%
\[
P_{b}(H_{1/2})\cdot P_{b}(H_{1/2})\subset P_{b}(H_{1/2})\text{\hspace
{0.5cm}and\hspace{0.5cm}}P_{b}(H)\cdot P_{b}(H)\subset P_{b}(H),
\]
so that all the results of Section \ref{Wiener}\ apply to both $H_{1/2}$ and
$H$.
%Again,
In particular the conclusion of Wiener's theorem holds on
%$H_{1/2}$\ and
$H$,
%\
and again on $H_{1/2}$, for all even $p\geq1$. In Section
\ref{WienerCountableCompact} we will show that the same conclusion
%Wiener's theorem
holds on both $H_{1/2}$\ and$\ H$ for all $p\in\lbrack1,\infty].$
\end{remark}

\subsection{Localizing properties of functions}

Functions on $H$ are positive definite if and only if their restrictions to
each subhypergroup $U_{N}$ are positive definite. The same is true for
continuity of functions on $H$. If $g\in C_{c}(K)$ then the convolution
$g^{\ast}\ast g$ vanishes outside $U_{N}$ for some integer $N$. It follows
that a (locally integrable) function is of positive type on $H$ if and only if
the restriction of that function to each $U_{N}$ is of positive type.
%See
Lemma \ref{TrivialExtension} below provides
%for
a converse to this.
%\subsection{Equivalence of definitions, absolute value and products\label{Equivalence}}

It is again clear that every $\ell^{1}$ sum of characters (including
$\chi_{-\infty}$) with non-negative coefficients is continuous, bounded and
positive definite. Conversely, given a function $f$ in $P(H)$, denote its
restriction to the subhypergroup $U_{N}$ by $f|_{U_{N}}$.
%That restriction is continuous and positive definite
%on~$U_N$.
Then $f|_{U_{N}}$\ is bounded as $U_{N}$\ is compact, and by
\cite{Bloom/Heyer}, Lemma 4.1.3g,%
\[
\left\Vert f|_{U_{N}}\right\Vert _{\infty}=f|_{U_{N}}\left(  \infty\right)
=f\left(  \infty\right)
\]
for all $N\in\mathbb{Z}.$ It follows that $f$ is bounded on $H$, so then by
Bochner's theorem again there exist non-negative $\alpha_{-\infty}$ and
$\alpha_{j},$ $j\in\mathbb{Z}$ with $\sum_{j}\alpha_{j}<\infty$ such that
$f=\alpha_{-\infty}\chi_{-\infty}+\sum_{j}\alpha_{j}\chi_{j}$, and hence
$\Vert f\Vert_{\infty}=f(\infty)$ and $P(H)=P_{b}(H)$.

%We apply the following
%We now show that the absolute-value function also operates
%on another class of positive definite functions. We show later
%in this section that neither class is included in the other.
%Another
%A
%immediate
%easy consequence of localizing in Section \ref{WienerCountableCompact}.
%is the following:

%A consequence of localizing, using Proposition \ref{operate}  and the lines after (\ref{Compact %neighbourhoods}), is the following Proposition, which

The following proposition is a corollary of Proposition \ref{operate}, using
localization and the lines after (\ref{Compact neighbourhoods}), and will
prove useful in Section \ref{WienerCountableCompact}.

\begin{proposition}
\label{Unbounded}
%Let $1 \le p < \infty$.
Let $1\leq p<\infty$. Suppose that $f:H$
%_{\frac{1}{2}}
$\rightarrow\mathbb{C}$ is $p$-integrable in a neighbourhood
%$U$
of the identity.
%$e$.
If $f$ is of positive type
%and $1 \le p < \infty$,
then so is $\left\vert f\right\vert ^{p}$. In particular, if $f\in P\left(
H\right)  \ $then $\left\vert f\right\vert ^{p}\in P\left(  H\right)  $.
%If $f$ is of positive type on $H$ and
%locally
%$p$-integrable in a nieghbourhood of the identity, then $|f|^p$ is also of positive type. In particular, if $f\in P(H)$ then $\left \vert f\right \vert^p \in P(H)$.

\end{proposition}

%\begin{corollary}
%\label{WienerOnH}
%The analogue of inequality (\ref{power}) for $p=1$, that is
%\[
%\left \Vert f\, \tau_{x}\mathbf{1}_{U}\right \Vert _{1}
%\leq \left \Vert f\left(\tau_{x}\mathbf{1}_{U}\right)  ^{1/p}\right \Vert _{p}
%\leq C_{U}%^{1/p}%
%\left \Vert f\, \mathbf{1}_{U}\right \Vert _{1}%
%\]
%Wiener's theorem
%holds on $H$ for locally integrable functions of positive type.
%In the case $p=1$, Wiener's theorem holds on $H$ for functions in $\left(L^{1},\ell^{\infty}\right)  \left(  H\right)  $.
%\end{corollary}

%\begin{proof}
%Use localization and
%See the proof of
%Corollary \ref{Wiener's Theorem}.
%\end{proof}

\begin{lemma}
\label{TrivialExtension} Extend a function of positive type on the hypergroup
$U_{N}$ to all of $H$ by making it vanish
%, consider its extension to $H$ that vanishes
outside $U_{N}$. That
%Then that
extension is of positive type on $H$. In particular, the extension by zero of
a function in $P(U_{N})$ is in $P(H)$.
\end{lemma}

\begin{proof}
Denote the original function by $f_{N}$ and its
%trivial
extension by $f$. Since $f_{N}$ is locally integrable and $U_{N}$ is compact,
$f_{N} \in L^{1}(U_{N})$ and $f \in L^{1}(H)$.

To apply Remark \ref{Positive type}, let $\chi$ be a character on $H$. Then
its restriction $\chi|_{U_{N}}$ to $U_{N}$ is a character on $U_{N}$, and
$\hat{f}(\chi)=\widehat{f_{N}}(\chi|_{U_{N}})$. Since every character on
$U_{N}$ has positive Plancherel measure, $\widehat{f_{N}}(\chi|_{U_{N}})\geq
0$, and hence $\hat{f}(\chi)$ is also nonnegative.
\end{proof}

\subsection{Discrete amalgam norms}

\label{DiscreteAmalgamNorms}

We used the amalgam norm
\begin{equation}
\left\Vert f\right\Vert _{p,\infty}=\sup_{n}\left(  \frac{1}{\omega_{\alpha
}\left(  I_{n}\right)  }\int_{n-1}^{n}\left\vert f\right\vert ^{p}%
\,d\omega_{\alpha}\right)  ^{\frac{1}{p}} \label{SupOfLpAverages}%
\end{equation}
to state Theorem \ref{transforms} for Bessel-Kingman hypergroups. Consider the
corresponding norm
%functional
on $H$. Given the division by the mass $\omega_{\alpha}(I_{n})$ here, the
integral above should run over the interval $I_{n}$. In $H$ that coincides
with the set $\{n-1\},$ with the curious outcome that
\begin{equation}
\Vert f\Vert_{p,\infty}=\sup_{n}{|f(n-1)|}=\sup_{n}{|f(n)|}=\Vert
f\Vert_{\infty} \label{natural}%
\end{equation}
no matter what $p$ is.

When $p<\infty$, there are compactly supported functions in $L^{p}(H)$ that
tend to $\infty$ at $\infty$. Any such function $f$ has the property that
\begin{equation}
\sup_{n}\left(  \int\left\vert f\right\vert ^{p}\tau_{n}\mathbf{1}%
_{U}\,d\omega\right)  ^{\frac{1}{p}}<\infty\label{SupLpNorms}%
\end{equation}
for each compact neighbourhood $U$ of $\infty$ even though $\Vert
f\Vert_{p,\infty}=\infty$. So the norm $\Vert\cdot\Vert_{p,\infty}$\ is not
equivalent to the one given in (\ref{SupLpNorms}). But the modified norm%
\begin{equation}
\Vert f\Vert_{p,\infty}^{\ast}=\max\left\{  \left\Vert f\mathbf{1}%
_{H\backslash U_{0}}\right\Vert _{p,\infty},\left\Vert f\mathbf{1}_{U_{0}%
}\right\Vert _{p}\right\}  , \label{ModifyLpinfty}%
\end{equation}
where $U_{0}$
%$U_{0}(=\mathbb{Z}_{+}\cup \left \{  \infty \right \}  )$
can be replaced by any compact neighbourhood of $\infty,$ is equivalent to the
norm in (\ref{SupLpNorms}).
%by taking~$\sup_n \left\|f\tau_{n}\mathbf{1}_{U}\right\|_{p}$.

Different
%To confirm this, recall that
%we showed after Corollary \ref{Amalgam equivalence} that
%different
choices of $U$ in \eqref{SupLpNorms} give norms that are equivalent to each
other, by the argument just after Corollary \ref{Amalgam equivalence}. Similar
reasoning applies to \eqref{ModifyLpinfty}, and
%So
it suffices to prove the equivalence between the latter and the norm in
\eqref{SupLpNorms} when $U=U_{0}$. Split
%Since $U_{k}\subset U\subset U_{\ell}$ for suitable $k,\ell \in \mathbb{Z}$, it suffices to prove the norm equivalence with $U=U_{i}$ $,$ $i\in \mathbb{Z}$ in (\ref{SupLpNorms}).\
%We split
the calculation of the supremum in (\ref{SupLpNorms}) into two cases
corresponding to different instances
%parts
of (\ref{NoncompactConvolution}). For $n<0$ we have $\tau_{n}\mathbf{1}%
_{U_{0}}=2^{n+1}\mathbf{1}_{\{n\}},$ so that%
\[
\left(  \int\left\vert f\right\vert ^{p}\tau_{n}\mathbf{1}_{U_{0}}%
\,d\omega\right)  ^{\frac{1}{p}}=\left\vert f\left(  n\right)  \right\vert .
%2^{-\frac{i}{p}}.
\]
%For $n<i$ we have $\tau_{n}\mathbf{1}_{U_{i}%
%}=2^{n-i+1}\mathbf{1}_{\{n\}},$ so that%
%\[
%\left(  \int \left \vert f\right \vert ^{p}\tau_{n}\mathbf{1}_{U_{i}}%
%\,d\omega \right)  ^{\frac{1}{p}}=\left \vert f\left(  n\right)  \right \vert2^{-\frac{i}{p}}.
%\]
For $n\geq0$
%For $n\geq i$
we obtain $\tau_{n}\mathbf{1}_{U_{0}}=\mathbf{1}_{U_{0}},$
%$\tau_{n}\mathbf{1}_{U_{i}}=\mathbf{1}_{U_{i}},$
and this gives%
\[
\left(  \int\left\vert f\right\vert ^{p}\tau_{n}\mathbf{1}_{U_{0}}%
\,d\omega\right)  ^{\frac{1}{p}}=\left\Vert f\mathbf{1}_{U_{0}}\right\Vert
_{p}.
\]
%\[
%\left(  \int \left \vert f\right \vert^{p}\tau_{n}\mathbf{1}_{U_{i}}%
%\,d\omega \right)  ^{\frac{1}{p}}=\left \Vert f\mathbf{1}_{U_{i}}\right \Vert_{p}.
%\]
%Since $U_{i}$ and $U_{0}$ differ by finitely many points only, and since in finite dimensional spaces all norms are equivalent, we obtain the asserted equivalence.
By formula \eqref{natural}, the norms in \eqref{SupLpNorms} and
\eqref{ModifyLpinfty} coincide when $U=U_{0}$.

When $1\leq q<\infty$, let
\begin{equation}
\Vert f\Vert_{p,q}^{\ast}=\left\{  \left\Vert f\mathbf{1}_{H\backslash U_{0}%
}\right\Vert _{p,q}^{q}+\left\Vert f\mathbf{1}_{U_{0}}\right\Vert _{p}%
^{q}\right\}  ^{1/q} \label{modifypq}%
\end{equation}
where
\[
\left\Vert f\mathbf{1}_{H\backslash U_{0}}\right\Vert _{p,q}\equiv\left\{
\sum_{n<0}\omega(\{n\})|f(n)|^{q}\right\}  ^{1/q}%
\]
actually doesn't depend on $p$. Whenever $1\leq p,q\leq\infty$, denote the
space of functions $f$ on $H$ for which $\Vert f\Vert_{p,q}^{\ast}<\infty$ by
$(L^{p},\ell^{q})(H)$.

On $H$, the structure of these spaces is simpler than it is on the real line
or on the Bessel-Kingman hypergroups. A function belongs to $(L^{p},\ell
^{q})(H)$ if and only if both its restriction to the set $U_{0}$ belongs to
$L^{p}$ and its restriction to the complement of $U_{0}$ belongs to $L^{q}$.
%$\ell^{q}$.

Since $\omega(U_{0})=1$, the restriction to $U_{0}$ then belongs to $L^{r}$
for all $r\leq p$. Since each point in the complement of $U_{0}$ has mass at
least $1$, the restriction to the complement then belongs to $L^{r}$ for all
$r\geq q$.
%It follows
Extend those restrictions by $0$ to see that $(L^{p},\ell^{q})(H)$ contains
the same functions as $L^{p}(H)+L^{q}(H)$ when $p\leq q$, and the same
functions as $L^{p}(H)\cap L^{q}(H)$ when $p\geq q$.
%, while $(L^p, \ell^q)(H)$ contains the same functions .

\subsection{Fourier transforms}

The norms $\Vert\cdot\Vert_{p,q}^{\ast}$ have good properties relative to
Fourier transforms (see below). Define $\Vert\cdot\Vert_{p,q}$ on
$\overset{\wedge}{H}$ as for $H$ just by replacing $\omega$ by $\pi$. Let
\[
U_{0}^{\perp}\equiv\{n\in\overset{\wedge}{H}:n\leq0\}
\]
and use $U_{0}^{\perp}$\ and its complement in\ $\overset{\wedge}{H}$ to
define\ $\Vert\cdot\Vert_{p,q}^{\ast}$ as in equations\ (\ref{ModifyLpinfty})
and\ (\ref{modifypq}). We have the following counterpart of Theorem
\ref{transforms}.

\begin{theorem}
\label{P(H)counterpart} The following statements are equivalent for a (locally
integrable) function $f$ of positive type on the hypergroup $H$:

\begin{enumerate}
\item $f$ is square integrable in a neighbourhood of the identity;

\item $f$ is the (inverse) transform of a function in the space $\left(
L^{1},\ell^{2}\right)  (\overset{\wedge}{H});$

\item $f\in\left(  L^{2},\ell^{\infty}\right)  \left(  H\right)  .$
\end{enumerate}
\end{theorem}

\begin{proof}
Again this follows if
%matters reduce to checking that
%We omit most of the proof, except for the fact that
the Fourier transform extends from $L^{1}(H)\cap L^{2}(H)$ to have
%has
appropriate mapping properties between suitable amalgam spaces, that is,
\begin{equation}
\text{if}\quad\Vert f\Vert_{p,q}^{\ast}<\infty,\quad\text{where}\quad1\leq
p,q\leq2,\quad\text{then}\quad\Vert\overset{\wedge}{f} \Vert_{q^{\prime
},p^{\prime}}^{\ast}<\infty. \label{HY}%
\end{equation}
By the observations at the end of Section \ref{DiscreteAmalgamNorms}, this is
equivalent to checking, when $1 \le p, q \le2$, that if $f \in L^{p}%
(H)+L^{q}(H)$ then $\hat f \in L^{q^{\prime}}(\hat H)+L^{p^{\prime}}(\hat H)$,
and the same for $L^{p}(H)\cap L^{q}(H)$ and $L^{q^{\prime}}(\hat H)\cap
L^{p^{\prime}}(\hat H)$.
%checking the corresponding statement for $L^p(H)\cap L^q(H)$.
%when $1 \le q < p \le 2$.
Both parts follow immediately from the Hausdorff-Young theorem
\cite{Degenfeld-Schonberg} for hypergroups.
\end{proof}

\begin{remark}
\label{Constant1} In fact,
%\left
$\Vert\hat{f}$
%\right
$\Vert_{q^{\prime},p^{\prime}}^{\ast}\leq\Vert f\Vert_{p,q}^{\ast}$ in all
these cases. Complex interpolation again reduces matters to proving this in
the extreme cases where $(p,q)$ is one of $(1,1),(2,2),(1,2)$ and $(2,1)$. The
first two
%of these
%extreme
cases are true because
\[
\Vert\overset{\wedge}{f}\Vert_{\infty}\leq\Vert f\Vert_{1}\quad\text{and}%
\quad\Vert\overset{\wedge}{f}\Vert_{2}=\Vert f\Vert_{2}.
\]
The corresponding estimates in the
%The
other two extreme cases follow from each other by duality as in Section
\ref{FourierTransform}.
%\ref{hypergroups}.

We elect to confirm the case where $(p,q)=(2,1)$ and $(q^{\prime}, p^{\prime})
= (\infty, 2)$. Split $f$ as $f_{1} + f_{2}$, where $f_{2} = {f}%
\mathbf{1}_{U_{0}}$ and $f_{1}$ vanishes on $U_{0}$. Since
\[
\|f\|^{*}_{2, 1} = \|f_{1}\|_{1} + \|f_{2}\|_{2},
\]
it suffices to show that $
%\left
\| \widehat{f_{1}}
%\right
\|^{*}_{\infty, 2} \le\|f_{1}\|_{1}$ and $
%\left
\| \widehat{f_{2}}
%\right
\|^{*}_{\infty, 2} \le\|f_{2}\|_{2}$.
%respectively.

Note that $\widehat{f_{1}}(n)=0$ for all $n>0$, since the support of $f_{1}$
is disjoint from that of $\chi_{n}$ when $n>0$. So
%Then
%Because of this,
$
%\left
\Vert\widehat{f_{1}}
%\right
\Vert_{\infty,2}^{\ast}$ simplifies to become $
%\left
\Vert{\widehat{f_{1}}}\mathbf{1}_{U_{0}^{\perp}}
%\right
\Vert_{\infty}$, and
\[
%\left
\Vert\widehat{f_{1}}
%\right
\Vert_{\infty,2}^{\ast}\leq
%\left
\Vert\widehat{f_{1}}
%\right
\Vert_{\infty}\leq\Vert f_{1}\Vert_{1}\quad\text{as required.}%
\]
Note also that the characters $\chi_{n}$ with $n\leq0$ are all equal to $1$ on
the set $U_{0}$, making $\widehat{f_{2}}$ constant on the set $U_{0}^{\perp}$.
Then $
%\left
\Vert\widehat{f_{2}}\mathbf{1}_{U_{0}^{\perp}}
%\right
\Vert_{\infty}=
%\left
\Vert{\widehat{f_{2}}\mathbf{1}_{U_{0}^{\perp}}}
%\right
\Vert_{2}$ since $\pi\left(  U_{0}^{\perp}\right)  =1$. Expand $
%\left
\Vert\widehat{f_{2}}
%\right
\Vert_{\infty,2}^{\ast}$ as
\begin{align*}
\{(\Vert\widehat{f_{2}}\mathbf{1}_{U_{0}^{\perp}}\Vert_{\infty})^{2}%
+\sum_{n>0}\pi(\{n\})|\widehat{f_{2}}(n)|^{2}\}^{\frac{1}{2}}  &
=\{(\Vert\widehat{f_{2}}\mathbf{1}_{U_{0}^{\perp}}\Vert_{2})^{2}+\sum_{n>0}%
\pi(\{n\})|\widehat{f_{2}}(n)|^{2}\}^{\frac{1}{2}}\\
&  =\Vert\widehat{f_{2}}\Vert_{2}=\Vert f_{2}\Vert_{2}\quad\text{as required.}%
\end{align*}

\end{remark}

\subsection{Wiener's theorem for all exponents\label{WienerCountableCompact}}

We will show that versions of Wiener's theorem hold on $H$ for all exponents
in the interval $[1,\infty]$, but we first note that Lemma \ref{Property*} can
be sharpened in the case of this hypergroup:
%hypergroup.

\begin{remark}
\label{Sharpening}For $U=U_{N}$ we may choose the neighbourhood $V$ in the
proof of Lemma \ref{Property*} to be $U_{N}$ as well. Instead of inequality
(\ref{nesting}) we obtain
\[
h:=\mathbf{1}_{V}^{\sim}\ast\mathbf{1}_{V}={\omega(U_{N})}\mathbf{1}_{U_{N}}.
\]
The next step in that proof then works with the singleton $x_{1}=\{e\}$, the
parameter $\lambda_{1}=1/\omega(U_{N})$ and the measure $\nu=\lambda
_{1}\varepsilon_{e}$. The long chain of equalities and inequalities there ends
with the quantity
\[
%\le
\Vert\nu\Vert\int hg\,d\omega_{K}.
\]
For the special choice of $h$ above, this is
\[
\Vert\nu\Vert\left\{  \omega(U_{N})\int_{U_{N}}g\,d\omega_{K}\right\}
\]
which gives the conclusion of Lemma \ref{Property*} with
\[
C_{U_{N}}=\Vert\nu\Vert{\omega(U_{N})}=1.
\]

\end{remark}

%We now consider estimates for amalgam norms over the whole hypergroup.

It follows that Corollary \ref{special case}
%\ref{WienerOnH}
holds with $C_{U} = 1$ when $U = U_{N}$. Since the proof of that corollary
only requires that $|f|^{p} \in P_{b}(K)$,
%be positive definite,
Proposition \ref{Unbounded} yields the conclusion of the corollary
%with $C_U = 1$
for all exponents $p$ in the interval $[1, \infty)$, again with $C_{U} = 1$ if
$U = U_{N}$ for some $N$.
%Similar reasoning applies, without the requirement that $C_U = 1$, when $U$ is not one of the sets $U_N$.
The
%Indeed, the
proof of Corollary \ref{bad} shows, for such exponents $p$,
%for any exponent $p$ in the interval $[1, \infty)$
that if inequality \eqref{(c2)} holds for all functions $f$ in $P_{b}(K)$,
then the inequality holds with the same constant $C_{U}$ for all integrable
%$p$-integrable
functions $f$
%in $L^1(K)$
that are of positive type.
%need not be

%We now consider estimates for amalgam norms over the whole hypergroup.

\begin{theorem}
\label{LocalGlobal} Let $p\in\left[  1,\infty\right]  $ and $f$ be a function
of positive type on H. Then
\begin{equation}
\Vert f\Vert_{p,\infty}^{\ast}=\Vert f\Vert_{p,\infty,U_{0}}=\left\Vert
f\mathbf{1}_{U_{0}}\right\Vert _{p}. \label{sharp}%
\end{equation}
For
%As a consequence, for
a general relatively compact neighbourhood $U$ of the identity there are
constants $C_{U}$ and $C_{U}^{\prime}$ (independent of $p$) such that
%For every neighbourhood $U$ of the identity in $H$ there is a constant $C_{U}>0$ (independent of $p$) such that%
\begin{equation}
\Vert f\Vert_{p,\infty,U}\leq C_{U}\left\Vert f\mathbf{1}_{U}\right\Vert
_{p}\text{\hspace{0.5cm}and\hspace{0.5cm}}\Vert f\Vert_{p,\infty}^{\ast}\leq
C_{U}^{\prime}\left\Vert f\mathbf{1}_{U}\right\Vert _{p} \label{globalamalgam}%
\end{equation}
for all (locally integrable) functions $f$ of positive type.
%functions~$f$ in~$P(H)$ and all
%in $L^{1}\left(H\right)$.

\end{theorem}

\begin{corollary}
\label{WienerTheorem}Let $p\in\left[  1,\infty\right]  $. For every relatively
compact neighbourhood $U$ of the identity in $H$ and every compact subset $V$
of $H$ there is a constant $C_{U,V}$
%> 0$
%$C_{U}>0$
(independent of $p$) such that%
\begin{equation}
\Vert f\mathbf{1}_{V}\Vert_{p}\leq C_{U,V}\left\Vert f\mathbf{1}%
_{U}\right\Vert _{p} \label{OtherCompactSet}%
\end{equation}
%\begin{equation}
%\left \Vert f\left(  \tau_{x}\mathbf{1}_{U}\right)  ^{1/p}\right \Vert _{p}\leq C_{U}\left \Vert f\mathbf{1}_{U}\right \Vert _{p} \label{localinequality}%
%\end{equation}
%for all points $x$ in $H$ and
for all (locally integrable) functions $f$ of positive type.
%for all locally functions~$f$ in~$P(H)$.
%This also holds for all locally integrable functions $f$ of positive type.
%in $L^{1}\left(H)$.
%If such a function $f$
%or of a function $f$ in~$P(H)$,
%is $p-$integrable on some neighbourhood $U$ of the identity, and if $V$ is a compact subset of $H$, then $f$ is $p-$integrable on $V$, and
%$f\mathbf{1}_{U}\in L^{p}\left(  H\right) \in L^{p}\left(
%H\right)  $ with%
%\begin{equation}
%\Vert f\mathbf{1}_{V}\Vert_{p}\leq C_{U,V}\left \Vert f\mathbf{1}%
%_{U}\right \Vert _{p} \label{OtherCompactSet}%
%\end{equation}
%with $C_{U,V}$\ independent of $p$.

\end{corollary}

\begin{corollary}
\label{CompactCase} Let $p\in\left[  1,\infty\right]  $. For every
neighbourhood $U$ of the identity in the compact hypergroup $H_{1/2}$ there is
a constant $C_{U}$\ (independent of $p$) such that%
\begin{equation}
\Vert f \Vert_{p}\leq C_{U}\left\Vert f\mathbf{1}_{U}\right\Vert _{p}
\label{global}%
\end{equation}
%\begin{equation}
%\left \Vert f\right \Vert _{p}\leq C_{U}\left \Vert f\mathbf{1}_{U}\right \Vert_{p} \label{localinequalitycompact}%
%\end{equation}
for all functions $f$ of positive type.
%that are $p$-integrable on $U$.
%for all functions $f$ of positive type.
%for all functions $f$ of positive type in $L^{1}\left(H_{1/2}\right)$.

\end{corollary}

\begin{proof}
[Proofs]As in Corollary \ref{Infinity norm}, the cases where $p=\infty$ follow
from those where $p < \infty$. In the latter cases, there is nothing to prove
unless $\left\Vert f\mathbf{1}_{U}\right\Vert _{p} < \infty$. Restricting $f$
to various subhypergroups $U_{N}$ and extending those restrictions by $0$
%trivially
then reduces matters to cases where $f$ has compact support and is therefore
$p$-integrable, hence integrable.

The
%To obtain
%Start with
%Theorem \ref{LocalGlobal},
%and
%recall that the
first equality in \eqref{sharp} was shown, when $1 \le p < \infty$, in the
lines following \eqref{ModifyLpinfty}. For the second equality, it is clear
from the definition of $\Vert f\Vert_{p,\infty, U_{0}}$
%and  $\Vert f\Vert_{p, \infty, U_0}$
that it is
%they are
no smaller than $\left\Vert f\mathbf{1}_{U_{0}}\right\Vert _{p}$. The opposite
inequality $\Vert f\Vert_{p,\infty, U_{0}} \leq\left\Vert f\mathbf{1}_{U_{0}%
}\right\Vert _{p}$ holds
%when $p=1$,
because of the discussion after Remark \ref{Sharpening}.
%yields that $\Vert f\Vert_{p, \infty, U_0} \le \left\Vert f\mathbf{1}_{U_0}\right \Vert_{p}$.
%in that case.
%the opposite inequality
%that
%$\Vert f\Vert_{p,\infty}^{\ast} \leq \left\Vert f\mathbf{1}_{U_0}\right \Vert_{p}$
%When $p>1$,
%note
%H\"older's inequality yields
%that
%$\left \Vert f\mathbf{1}_{U_0}\right \Vert_{1}
%\le \left \Vert f\mathbf{1}_{U_0}\right \Vert_{p}$,
%since $\pi(U_0) =1$. Therefore,
%\[
%\sup_{n < 0} |f(n)| = \left \Vert f\mathbf{1}_{H\backslash U_{0}}\right \Vert _{1,\infty} \leq \Vert f\Vert_{1,\infty}^{\ast} = \left \Vert f\mathbf{1}_{U_0}\right \Vert_{1} \le \left \Vert f\mathbf{1}_{U_0}\right \Vert_{p}.
%\]
%Since the supremum on the left above is also equal to $\left \Vert f\mathbf{1}_{H\backslash U_{0}}\right \Vert _{p,\infty}$,
%\[
%\left \Vert f\mathbf{1}_{H\backslash U_{0}}\right \Vert _{p,\infty}
%\le \left \Vert f\mathbf{1}_{U_0}\right \Vert_{p},
%\quad\textnormal{
%and
%so that
%}\quad
%\Vert f\Vert_{p,\infty}^{\ast}
%= \Vert f\Vert_{p, \infty, U_0}
%= \left \Vert f\mathbf{1}_{U_0}\right \Vert_{p}
%\quad\textnormal{as required.}
%.
%\]
%the latter is also majorized by $\left \Vert f\mathbf{1}_{U_0}\right \Vert_{p}$.
The same discussion yields the first inequality in line \eqref{globalamalgam}.
The second inequality then follows
%The two inequalities in line \eqref{globalamalgam} follow from each other
by the equivalence of the norms $\Vert\cdot\Vert_{p,\infty, U}$ and
$\Vert\cdot\Vert_{p,\infty}^{\ast}$.
This completes the proof of Theorem \ref{LocalGlobal}.

%The first part of
For Corollary \ref{WienerTheorem},
%consider
use
%then follows from
the chain of inequalities
\[
\Vert f\mathbf{1}_{V}\Vert_{p}\leq\Vert f\Vert_{p,\infty,V}\leq C_{U,V}%
^{\prime}\Vert f\Vert_{p,\infty,U}\leq C_{U,V}^{\prime}C_{U}\left\Vert
f\mathbf{1}_{U}\right\Vert _{p},
\]
where the first step uses the definition of $\Vert\cdot\Vert_{p,\infty,V}$,
the second step uses the equivalence of that norm with $\Vert\cdot
\Vert_{p,\infty,U}$ and the last step uses the first inequality in
\eqref{globalamalgam}.
%, inequality (\ref{localinequality}), then follows
%from
%the equivalence of the norm $\Vert \cdot \Vert_{p,\infty}^{\ast}$ with the one given by $\sup_{x\in H}\left \Vert f\, \left(  \tau_{x}\mathbf{1}%
%_{U}\right)  ^{1/p}\right \Vert _{p}$. For the second part, choose points $x_{i}\in H$ and constants $c_{i}>0$ such that%
%\[
%\sum_{i=1}^{r}c_{i}\tau_{x_{i}}\mathbf{1}_{U}\geq \mathbf{1}_{V},
%\]
%multiply this by $\left \vert f\right \vert ^{p}$ and use (\ref{localinequality}%
%) to obtain (\ref{OtherCompactSet}) with constant $C_{U,V}=\left(  \sum_{i=1}^{r}c_{i}\right)  ^{1/p}C_{U}$.
Corollary \ref{CompactCase} follows because extending
%the function
$f$ by $0$
%to take the value $0$ on the rest of $H$
gives a function of positive type on $H$.
%in $L^{1}(H)$.

\end{proof}

%

%TCIMACRO{\QTR{UserInput}{\TeXButton{Comment}{\begin{comment}}%
%\TeXButton{Comment}{\end{comment}}}}%
%BeginExpansion
\begin{comment}\end{comment}%
%EndExpansion

\begin{remark}
The first inequality in
%Inequality
\eqref{globalamalgam} provides an upper bound for $\|f\|_{p, \infty, U}$ in
terms of $\left\Vert f\mathbf{1}_{U}\right\Vert _{p}$. When $p<\infty$, there
is no such general bound
%such  a bound is  not possible
for $\|f\|_{p}$. Indeed, since~$\sum_{n=-\infty}^{\infty}\omega(n) = \infty$,
the constant function $\mathbf{1}$ trivially belongs to the set $P(H)$ but to
none of the spaces $L^{p}(H)$ with $0<p<\infty$.
%because~$\sum_{n=-\infty}^\infty \omega(n) = \infty$.
%that function belongs to.
%the fact that $\omega{\{n\}}\rightarrow \infty$ as $n\rightarrow-\infty$ prevents any passage from upper bounds for c

\end{remark}

%\subsection{Examples}

%The constant function $\mathbf{1}$ trivially belongs to the set $P(H)$ but to none of the spaces $L^{p}(H)$ with $0<p<\infty$. So the norm on the left in inequality (\ref{localinequality}) cannot be replaced by $\|f\|_p$.

%A simple example of a function of positive type that belongs to $L^{1}(H)$ but not to $P(H)$
%of Proposition~\ref{GivenFunction} but not those
%of Proposition~\ref{GivenCoefficients} is%
%\begin{equation}
%g=\sum_{n\geq0}\mathbf{1}_{U_{n}}.\label{sumsets}%
%\end{equation}
%This series converges in $L^{p}(H)$ for all $p<\infty$, because the measures of the sets $U_{n}$ tend to $0$ so fast. The indicator functions of those sets all belong to $P(H)$, making $\overset{\wedge}{g}\geq0$ as well. The sum $g$ vanishes to the left of $0$, and $g(n)=n+1$ otherwise. So $g\notin P(H)$, and the restriction of $g$ to $U_{0}$ does not belong to $P(U_{0})$.


\begin{thebibliography}{9999}                                                                                             %


\bibitem[$1$]{BertrandiasDupuis}Bertrandias, J. P. and Dupuis, C.,
\textit{Transformation de Fourier sur les espaces }$\ell^{p}\left(
L^{p^{\prime}}\right)  $\textit{. }Ann. Inst. Fourier Grenoble \textbf{29}
(1978), 189 - 206.

\bibitem[$2$]{Bloom/Heyer}Bloom, Walter R. and Heyer, Herbert,
\textit{Harmonic analysis of probability measures on hypergroups}. De Gruyter
Studies in Mathematics, \textbf{20,} Walter de Gruyter, Berlin-New York, 1995.

\bibitem[$3$]{CowlingMedaPasquale}Cowling, Michael, Meda, Stefano and
Pasquale, Roberta, \textit{Riesz potentials and amalgams}, Annales de
l'institut Fourier \textbf{49} (1999), 1345 - 1367.

\bibitem[$4$]{Degenfeld-Schonberg}Degenfeld-Schonburg, Sina Andrea,
\textit{Multipliers for hypergroups: Concrete examples, applications to time
series}, PhD Thesis, Technische Universit\"{a}t, M\"{u}nchen, 2012 (http://d-nb.info/1031513728/34).

\bibitem[$5$]{DunklRamirez}Dunkl, Charles F. and Ramirez, Donald E., \textit{A
family of countably }$P_{\ast}-$\textit{hypergroups}, Trans. Amer. Math. Soc.
\textbf{202} (1975), 339 - 356.

\bibitem[$6$]{Feichtinger}Feichtinger, Hans, \textit{Banach convolution
algebras of Wiener type}, Colloq. Math. Soc. Janos Bolyai \textbf{35} (1980),
509 - 524.

\bibitem[$7$]{Fournier1}Fournier, John J. F., \textit{Local and global
properties of functions and their Fourier transforms}, T\^{o}hoku Math. J.
\textbf{49} (1997), 115 - 131.

\bibitem[$8$]{Fournier 2}Fournier, John J. F., \textit{On the Hausdorff-Young
theorem for amalgams, }Monatsh. Math\textit{. }\textbf{95} (1983), 117 - 135.

\bibitem[$9$]{FournierStewart}Fournier, John J. F. and Stewart, James,
\textit{Amalgams of }$L^{p}$\textit{\ and }$\ell^{q}$, Bull. Amer. Math. Soc.
(New Series) \textbf{13} (1985), 1 - 21.

\bibitem[$10$]{Holland}Holland, Finbarr, \textit{Harmonic analysis on amalgams
of }$L^{p}$\textit{\ and }$\ell^{q}$, J. London Math. Soc. (2) \textbf{10}
(1975), 295 - 305.

\bibitem[$11$]{Innig}Innig, Lukas, \textit{Diplomarbeit \"{u}ber}
\textit{einen Satz von Wiener auf der Bessel-Kingman Hypergruppe}, Diploma
Thesis, Universit\"{a}t Heidelberg, 2011.

\bibitem[$12$]{LeinBook}Leinert, Michael, \textit{Integration und Ma{\ss },}
Vieweg, 1995.

\bibitem[$13$]{Leinert}Leinert, Michael, \textit{On a theorem of Wiener}.
Manuscripta Math. \textbf{110} (2003), 1 - 12.

\bibitem[$14$]{Rains}Rains, Michael, \textit{On functions with non-negative
Fourier transforms}, Indian J. Math. \textbf{27} (1985), 41 - 48.

\bibitem[$15$]{Shapiro}Shapiro, H. S., \textit{Majorant problems for Fourier
coefficients}, Quart. J. Math. Oxford, \textbf{26} (1975), 9 - 18.

\bibitem[$16$]{Spector}Spector, Ren\'{e}, \textit{Une classe d'hypergroupes
d\'{e}nombrables}, C. R. Acad. Sci. Paris S\'{e}r. A-B, \textbf{281} (1975),
A105 - A106.

\bibitem[$17$]{Wainger}Wainger, S., \textit{A problem of Wiener and the
failure of the principle for Fourier series with positive coefficients}, Proc.
Amer. Math. Soc. \textbf{20} (1969), 16 - 18.

\bibitem[$18$]{Zeuner}Zeuner, Hansmartin, \textit{Duality of commutative
hypergroups}. Probability measures on groups, X (Oberwolfach, 1990), 467--488,
Plenum, New York, 1991.
\end{thebibliography}
\end{document}